\documentclass[a4paper,12pt]{article}

\usepackage{amsthm,amsmath,stmaryrd,bbm,hyperref,geometry,color}
\usepackage[utf8]{inputenc}
\usepackage{amssymb}
\usepackage[english]{babel}
\usepackage{graphicx}
\usepackage{amsfonts,amssymb}
\usepackage{verbatim}
\usepackage{enumitem}
\usepackage[all]{xy}

\newcommand{\po}{\left(}
\newcommand{\pf}{\right)}
\newcommand{\co}{\left[}
\newcommand{\cf}{\right]}
\newcommand{\cco}{\llbracket}
\newcommand{\ccf}{\rrbracket}
 
\newcommand{\T}{\mathbb T}

\newcommand{\N}{\mathbb N} 
\newcommand{\dd}{\text{d}}

\renewcommand{\1}{\mathbbm{1}}
\newcommand{\X}{\mathbf{X}}
\newcommand{\Y}{\mathbf{Y}}
\newcommand{\bX}{\mathbf{\overline{X}}}
\newcommand{\nv}[1]{#1}
\newcommand{\nvd}[1]{#1}
\newcommand{\nvdd}[1]{#1}

\newtheorem{thm}{Theorem}

\newtheorem{lem}[thm]{Lemma}

\newtheorem{prop}[thm]{Proposition}
\newtheorem{cor}[thm]{Corollary}

\title{Convergence of \nvd{ a particle approximation for the quasi-stationary distribution of a diffusion process}: uniform estimates in a compact soft case.}
\author{Lucas Journel, Pierre Monmarché}

\begin{document}
\maketitle

\abstract{We establish the convergences (with respect to the simulation time $t$; the number of particles $N$; the timestep $\gamma$) of a \nvd{Moran/Fleming-Viot type particle scheme} toward the quasi-stationary distribution of a diffusion on the $d$-dimensional torus, killed at a smooth  rate. In these conditions, quantitative bounds are obtained that, for each parameter ($t\rightarrow \infty$, $N\rightarrow \infty$ or $\gamma\rightarrow 0$) are independent from the two others.}

\section{Introduction}

\subsection{The problem}

Start from the diffusion on the $d$-dimensional periodic flat torus $\mathbb T^d$
\begin{eqnarray}\label{Eq-EDS}
\dd Z_t &=& b(Z_t) \dd t + \dd B_t
\end{eqnarray}
with $b\in\mathcal C^1(\mathbb T^d)$, where $(B_t)_{t\geqslant0}$ is a $d$-dimensional Brownian motion. Add a  killing rate $\lambda \in\mathcal C(\mathbb T^d)$ and, given a standard  exponential random variable $E$ independent from $(Z_t)_{t\geqslant 0}$, define the death time
\begin{eqnarray}\label{Eq-DeathTime}
T &=& \inf\left\{t\geqslant 0,\ E \leqslant \int_0^t \lambda(Z_s) \dd s\right\}\,.
\end{eqnarray}
Then a probability measure $\nu$  on $\mathbb T^d$ is said to be a quasi-stationary distribution (QSD) associated to the SDE \eqref{Eq-EDS} and the rate $\lambda$ if
\[ \mathcal Law( Z_0)= \nu \qquad \Rightarrow \qquad \forall t\geqslant 0,\ \mathcal Law(Z_t\ |\ T>t) = \nu\,.\]
In our case, there exists a unique QSD $\nu_*$ and, whatever the initial distribution $\eta_0$ of $Z_0$,
\[\mathcal Law(Z_t\ | \ T>t ) \ \underset{t\rightarrow \infty}\longrightarrow \ \nu_*\]
see e.g. \cite[Theorem 2.1]{ChampVilK2018}) \nv{ or Corollary \ref{Cor-tempslongContinuNonLin} below}.

\nv{The present work is dedicated to the proof of convergence of an algorithm designed to approximate $\nu_*$.} \nvd{This is classically done through a system of $N$ interacting particles  whose empirical measure converges to $\mathcal Law(Z_t\ | \ T>t )$ as $N\rightarrow\infty$, where killed particles are resurrected in a suitable way in order to keep constant the size of the system (while a naive Monte Carlo simulation would see the sample shrink along time).
This question has already been addressed by many authors in various contexts, see the discussion in Section~\ref{Sec:previous_works} below. Before introducing the algorithm, stating our results and comparing them with previous works, for now, let us simply highlight the main  specificities of the present work.

The first novelty is that we take into account the time-discretization of the continuous-time diffusion. That way, we establish error bounds between the theoretical target QSD and the empirical measure indeed obtained with an actual implementation of the algorithm. There are three sources of errors: first, the continuous-time SDE \eqref{Eq-EDS} has to be discretized with some time-step parameter $\gamma>0$. Second,  as will be detailed below, a non-linearity in the theoretical algorithm has to be approximated by a system of $N$ particles. This leads to the definition of an ergodic Markov chain whose invariant measure is close, in some sense, for large $N$, to the QSD of the time-discretization of the diffusion. But then this Markov chain is only run for a finite simulation time $t=m\gamma$, $m\in\mathbb N$. A third error term then comes from the fact that stationarity is not fully achieved. We will obtained quantitative error bounds in $\gamma$, $N$ and $t$.}

\nvd{A second specifity is that the bound obtained for each parameter will be uniform in the other two. For instance, the only other work in which the long-time convergence of the chain is proven to be, under some (restrictive) conditions, uniform in $N$, is \cite{CloezThai} in a finite state space. Besides, our work is quite close in spirit to this work of Cloez and Thai. The question of the dependency or uniformity of the estimates in other previous works will be further discussed in Section~\ref{Sec:previous_works}.}

\nvd{Finally, although it was not the primary motivation of the present work, it seems that the particular definition of the system of interacting particles considered here, in particular the rebirth mechanism, was not considered in previous works (where, basically, killed particles are resurrected at the position of one of the other particles). Our variant is initially motivated by the property stated in Proposition~\ref{Prop-EgaliteLoi} below, which has been indicated to the second author by Bertrand Cloez. Yet, this variant has the unintended advantage to be both discrete in time and non-failable, in the sense that it is well-defined for all times, even though all particles die simultaneously from time to time (see also \cite{Villemonais2} on this question).}

Note that we restrict the study to a compact state space. Moreover, we only consider soft killing at some continuous rate, and no hard killing which would correspond to the case where $T$ is the escape time from some sub-domain (see e.g. \cite{BCV,Guyaderhard}). Finally, as will be seen below, as far as the long-time behaviour of the process is concerned we will work in a perturbative regime, namely we will assume that the variations of $\lambda$ are small with respect to the mixing time of the diffusion \eqref{Eq-EDS} \nvd{ (while $\|\lambda\|_\infty$ itself is not required to be small)}. \nvd{These very restrictive conditions, which rule out many cases of practical interest, have to be considered in   light of our very strong results (Theorems~\ref{Theorem-Inter} and \ref{Theorem-Main} below and all the corollaries of Section~\ref{Sec-conclusion}, gathered in Figure~\ref{Figure_Corollaires}).} In fact, although already interesting by itself, this restricted framework can be thought as a toy model motivated in particular by the case that arises in the parallel replica algorithm \cite{LelievreLeBris}. In that case, $T$ is the escape time for \eqref{Eq-EDS} from a bounded metastable domain, so that the lifespan of the process is expected to be larger than its mixing time (and to depend little from the initial condition, given it is far enough from the boundary). Hence, the compact and perturbative assumptions are consistent with this objective. The restriction to smooth killing rate, however, is made to avoid additional difficulties in the hard case where, even in the metastable case, the probability  to leave the domain is high (and exhibits high variations) when the process is close to its boundary. \nvd{The initial motivation of the present study was to  test the general strategy of the proof (via coupling aguments) in a first simple case, with the goal of extending it later on to the metastable hard case by combining it with some Lyapunov arguments to control the variations of the killing rate near the boundary}. This study is postponed to future work.

\bigskip

\nvdd{This work is organized as follows. The algorithm and main results are presented in Section~\ref{SubSec-DefProcess}, and the relation with previous works is discussed in Section~\ref{Sec:previous_works}. Section~\ref{Sec-proofs} contains the proofs, and more precisely: a general coupling argument, which is the central tool for all our results,  is presented in Section~\ref{Subsec:basiccoupling}; the basic bounds in term of  $t\rightarrow +\infty$, $N\rightarrow +\infty$ and $\gamma\rightarrow 0$ are then  stated and proven respectively in Sections~\ref{SubSection-Tempslong}, \ref{Subsec:propchaos} and \ref{Sec-continu}; finally, these basic results are combined in Section~\ref{Sec-conclusion}, concluding the proofs of the main theorems and inducing  a number of corollaries.
}

\subsection*{Notations and conventions}

We respectively denote $\mathcal P(F)$ and $\mathcal B(F)$ the set of probability measures and of Borel sets of a Polish  space $F$. Functions on $\T^d$  are sometimes identified to $[0,1]^d$-periodic functions, and similar non-ambiguous identifications are performed, for instance if $x\in\T^d$ and $G$ is a $d$-dimensional standard gaussian random variable, $x+G$ has to be understood in $\T^d$, etc. A Markov kernel $Q$ on $F$ is  indiscriminately  understood as, first, a function from $F$ to $\mathcal P(F)$, in which case we denote $Q:x\mapsto  Q(x,\cdot)$ (where $Q(x,\cdot)$ denotes the probability $A\in\mathcal B(F)\mapsto Q(x,A)\in[0,1]$); second, a Markov operator on bounded measurable functions on $F$, in which case we denote $Q:f\mapsto Qf$ (where $Qf(x) = \int f(w) Q(x,\dd w)$); third, by duality, a function on $\mathcal P(F)$, in which case we denote $Q:\mu\mapsto \mu Q$ (so that $\mu (Qf) = (\mu Q)f$). In particular, $Q(x,\cdot)= \delta_x Q$ for $x\in F$. 
If $\mu \in \mathcal P(F)$ and $k\in\N_*$, we denote $\mu^{\otimes k}\in\mathcal P(F^k)$ the law of a $k$-uplet of independent random variables with law $\mu$. Similarly, if $Q$ is a Markov kernel on $F$, we denote $Q^{\otimes k}$ the kernel on $F^k$ such that $Q^{\otimes k}(x,\cdot) = Q(x_1,\cdot)\otimes\dots\otimes Q(x_k,\cdot)$ for all $x=(x_1,\dots,x_k)\in F^k$. 
We denote $\mathcal E(1)$ the exponential law with parameter 1, $\mathcal U(I)$ the uniform law on a set $I$ and $\mathcal N(m,\Sigma)$ the Gaussian law with mean $m$ and variance matrix $\Sigma$.   We use bold letters for random variables in $\T^{dN}$ and decompose them in $d$-dimensional coordinates, like $\X = (X_1,\dots,X_N)$ with $X_i\in\T^d$, or $\X_1 = (X_{1,1},\dots,X_{N,1})$.

\subsection{The algorithm and main result}\label{SubSec-DefProcess}

Starting from the diffusion \eqref{Eq-EDS} killed at time $T$ given by \eqref{Eq-DeathTime}, we introduce two successive approximations. The first is time discretization. For a given time step $\gamma>0$ and a sequence $(G_k)_{k\in\mathbb N}$ of independent random variables with law $\mathcal N(0,I_d)$, we consider the Markov chain on $\T^d$ given by $\tilde Z_0 = Z_0$ and
\begin{eqnarray}\label{Eq-EulerSchemeEDS}
\forall k\in\N\,,\qquad\tilde Z_{k+1} &=& \tilde Z_k + \gamma b(\tilde Z_k) + \sqrt{\gamma} G_k
\end{eqnarray}
and, given $E\sim \mathcal E(1)$ independent from $(G_k)_{k\in\mathbb N}$ and $Z_0$,
\[\tilde T \ = \ \inf\left\{t = n\gamma,\ n\in\mathbb N_*,\ E \leqslant \gamma \sum_{k=1}^n \lambda(\tilde Z_k) \right\}\,.\]
From classical results for Euler schemes of diffusions \nv{(see e.g. \cite{Euler})}, it is quite clear  that, for any   $A\in\mathcal B( \mathbb T^d)$ and all $t\geqslant 0$,
\[\mathbb P \po \tilde Z_{\lfloor t/\gamma\rfloor} \in A,\ \tilde T < t \pf \ \underset{\gamma \rightarrow 0}\longrightarrow \ \mathbb P \po Z_t \in A,\ T<t\pf\,,\]
 from which, for all $t\geqslant 0$,
\[\mathcal Law\po \tilde Z_{\lfloor t/\gamma\rfloor} \ |\ \tilde T < t\pf \ \underset{\gamma \rightarrow 0}\longrightarrow \ \mathcal Law\po  Z_{t} \ |\  T < t\pf \]
(we will prove this, see Corollary \ref{Cor-concl3} below). Note that, from the memoryless property of the exponential law, given a  sequence $(U_k)_{k\in\mathbb N}$ of independent variables uniformly distributed over $[0,1]$ and independent from $(G_k)_{k\in\mathbb N}$ and $Z_0$, then $((\tilde Z_n)_{n\in\N},\tilde T)$ has the same joint distribution as $((\tilde Z_n)_{n\in\N},\hat T)$ with
\[\hat  T \  = \ \inf\left\{t = n\gamma,\ n\in\mathbb N_*,\ U_n \leqslant p(\tilde Z_n)\right\}\]
where $p(z)  =1 - \exp(-\gamma \lambda(z))$ is the probability that, arriving at state $z$, the chain is killed.

A naive Monte Carlo sampler for the QSD would be to simulate $N$ independent copies of the chain \eqref{Eq-EulerSchemeEDS} killed with probability $z\mapsto p(z)$ and to consider after a large number of iterations the distribution of the copies that have survived. However, after a long time, most copies (possibly all) would have died and the estimator would be very bad. \nvd{To tackle this issue, we have to introduce a rebirth mechanism to reintegrate  dead particles in the system.}

Denote $K:\mathbb T^d \rightarrow \mathcal P(\mathbb T^d)$ the Markov kernel associated with the transition \eqref{Eq-EulerSchemeEDS}, i.e.
\[Kf(x) \ = \ (2\pi)^{-d/2}\int_{\mathbb R^d} f\po x + \gamma b(x) + \sqrt{\gamma} y\pf e^{-\frac12 |y|^2}\dd y\,.\]
For $\mu \in \mathcal P(\mathbb T^d)$, let $Q_\mu$ be the Markov kernel such that, for all $x\in \mathbb T^d$, $Q_{\mu}(x,\cdot)$ is the law of the random variable $X$ defined as follows. 
 Let $(X_k,U_k)_{k\in\N}$ be a sequence of independent random variables such that, for all $k\in\N$, $X_k$ and $U_k$ are independent, $U_k\sim \mathcal U([0,1])$ and $X_0 \sim K(x,\cdot)$ while, for $k\geqslant 1$, $X_k\sim \mu K$. Let $H=\inf\{k\in\N,\ U_k \geqslant p(X_k)\}$, and set $X=X_{H}$. Since $\lambda$ is bounded, $p$ is uniformly bounded away from 1 and thus $H$ is almost surely finite, so that $Q_\mu$ is well-defined.

In other words, a random variable $X\sim Q_\mu(x,\cdot)$ may be constructed through the following algorithm (in which \emph{new} means: independent from all the  variables previously drawned).
\begin{enumerate}
\item Draw  $X_{0} \sim \mathcal N\po x + \gamma b(x) , \gamma  I_d\pf$ and a new $U_0\sim \mathcal U([0,1])$.
\item If $U_0 \geqslant p(X_{0})$, set $X = X_{0}$ in $\mathbb T^d$ (in that case, we say the particle has moved from $x$ to $X_0$ without dying).
\item If $U_0 <p(X_{0})$ then set $i=1$ and, while $X$ is not defined, do:
\begin{enumerate}
\item Draw a new $X_{i}'$ distributed according to $\mu$, a new $X_{i} \sim \mathcal N\po X_{i}' + \gamma b(X_{i}') , \gamma I_d \pf$ and a new $U_i\sim \mathcal U([0,1])$.
\item  If $U_i \geqslant p(X_{i})$, set $X = X_{i}$ in $\mathbb T^d$  (in that case, we say the particle has died, resurrected at $X_{i}'$, moved to $X_{i}$ and survived).
\item  If $U_i < p(X_{i})$, set $i\leftarrow i+1$ (in that case, we say the particle has died, resurrected at $X_{i}'$, moved to $X_{i}$ and died again) and go back to step (a).
\end{enumerate}
\end{enumerate}

From this, we define a chain $(Y_k)_{k\in\mathbb N}$ as follows.  Set $Y_0 = Z_0$ and suppose that $Y_k$ has been defined for some $k\in\mathbb N$. Let $\eta_k = \mathcal Law(Y_k)$, and draw a new $Y_{k+1} \sim Q_{\eta_k}(Y_k,\cdot)$. 
This somewhat intricate definition is motivated by the following results (whose proof is postponed to Section \ref{Sec-proofs}):
\begin{prop}\label{Prop-EgaliteLoi}
For all $n\in\mathbb N$
\[\eta_n \ = \ \mathcal Law\po \tilde Z_n  \ | \ \tilde T \ \nv{>} \ n\gamma\pf\,.\]
\end{prop}
In particular, as $n\rightarrow \infty$, the law $\eta_n$ of $Y_n$ converges toward the QSD of $\tilde Z$. Unfortunately, it is impossible to sample $(Y_k)_{k\in\mathbb N}$ in practice since this would require to sample according to $\eta_k$ for any $k\in\mathbb N$. \nvd{This is a classical case of a time-inhomogeneous Markov chain which is interacting with its own law or, similarly, of a measure-valued sequence $(\eta_k)_{k\in\N}$ with a non-linear evolution. Such processes arise in many applications,  see e.g. \cite{DelMoralbouquin,DelMoralbouquin2} and references within.}  Motivated by the Law of Large Numbers, we are lead to a second approximation, which is to use mean-field interacting particles. For a fixed $N\in\N_*$ and for $x=(x_i)_{i\in\cco 1,N\ccf}\in\mathbb T^{dN}$, we denote
\[\pi(x) \ := \ \frac{1}{N}\sum_{i=1}^N \delta_{x_i}\ \in \mathcal P( \mathbb T^d)\]
the associated empirical distribution. Then we define the Markov operator $R$ on $\mathbb T^{dN}$ as
\[R\po x,\cdot \pf \ = \  Q_{\pi(x)}(x_1,\cdot)\otimes \dots \otimes Q_{\pi(x)}(x_N,\cdot)\,.\]
 In other words, a random variable $\Y\sim Q(x,\cdot)$ is such that the $Y_i$'s are independent with $Y_i \sim Q_{\pi(x)}(x_i,\cdot)$.  In order to specify the parameters involved, we will sometimes write $R_{N,\gamma} $ for $R$.



Let us informally describe the transitions of such a Markov chain $(\X_k)_{k\in\N}$: the $i^{th}$ particle follows the transition given by \eqref{Eq-EulerSchemeEDS} independently from the other particles until it dies. If it dies at a step $k\in\N_*$, then it is resurrected on another particle $X_{J,k-1}$ with $J$ uniformly distributed over $\cco 1,N\ccf$ (in particular and contrary to \nvd{most previous works on similar algorithms, $J=i$ is not excluded, although it doesn't change much since its probability vanishes as $N\rightarrow\infty$}) and immediatly performs a step of \eqref{Eq-EulerSchemeEDS}; if it dies again after this unique step, it is resurrected again and performs a new step, and so on until it is not killed after a resurrection and an Euler scheme step. Then this is the new value $X_{i,k}$ from which the particle follows again the transitions \eqref{Eq-EulerSchemeEDS} until its next death, etc.

Note that there is no problem of simultaneous death since at step $k$ the particles are resurrected on positions at step $k-1$, which are well-defined even if all particles die at once at step $k$.

It is easily seen that $R$ admits a unique invariant measure toward which the law of the associated Markov chain converges exponentially fast (in the total variation sense for instance), but a naive argument yields a convergence rate that heavily depends on $N$ (and possibly $\gamma$). Similarly, classical studies can be conducted for the limits $N\rightarrow \infty$ and $\gamma\rightarrow 0$  but again with estimates that are typically exponentially bad with respect to the total simulation time (see the references in Section~\ref{Sec:previous_works} or Propositions~\ref{Prop-chaos} and \ref{Prop-continu}). In the following we will focus on a somewhat perturbative regime under which we will establish estimates for each of these limits that are uniform with respect to the other parameters. Even for the continuous-time process (\nvd{corresponding to $\gamma =0$ , see Section~\ref{Sec-continu} for the definition}), such uniform results are new (see Corollaries \ref{Cor-tempslong-continu} and \ref{Cor-chaos-continu}).


Recall that the $\mathcal W_1$ Wasserstein distance between $\mu,\nu\in\mathcal P(\T^d)$ is defined by
\[\mathcal W_1\po\mu,\nu\pf \ = \ \inf\left\{\mathbb E\po |X-Y|\pf\, :\ X\sim \mu,\ Y\sim\nu\right\}\,.\]
\nvdd{
More generally, for $\rho$ a distance on some Polish space $F$, denote $\mathcal W_\rho$ the corresponding Wasserstein distance on $\mathcal P(F)$, defined by
\begin{eqnarray}\label{Eq-def-Wasserstein}
\mathcal W_\rho\po\mu,\nu\pf &=& \inf\left\{\mathbb E\po \rho(X,Y)\pf\,:\ X\sim \mu,\ Y\sim\nu\right\}\,.
\end{eqnarray}
 If $X\sim\mu$ and $Y\sim\nu$, we call $(X,Y)$ a coupling of $\mu$ and $\nu$. If $(X,Y)$ is a coupling for which the infimum in \eqref{Eq-def-Wasserstein} is attained, we say that it is an optimal coupling. From \cite[Corollary 5.22]{Villani}, such an optimal coupling always exists.

Our first main result is a long-time convergence rate uniform in $N$:
\begin{thm}\label{Theorem-Inter}
There exist $c_1,c_2,\gamma_0>0$ and a distance $\rho$ on $\T^d$ equivalent to the Euclidean distance,  that depend only on the drift $b$ and the dimension $d$, such that, if $\lambda$ is Lipschitz with a constant $L_\lambda$ and
\begin{eqnarray}\label{Eq-ConditionPerturb}
\kappa & := &  c_1 - c_2 L_\lambda e^{\gamma \|\lambda\|_\infty}\,,
\end{eqnarray}
then the following holds: for all $\gamma\in (0,\gamma_0]$, 	$N\in\N$ and all $\mu,\nu\in\mathcal P(\T^{dN})$, considering the distance $\rho_N(x,y)=\sum_{i=1}^N \rho(x_i,y_i)$ for $x,y\in \T^{dN}$,
\begin{align*}
\mathcal W_{\rho_N}\po \mu R_{N,\gamma},\nu R_{N,\gamma}\pf \ \leqslant \ \po 1-\gamma \kappa \pf\mathcal W_{\rho_N}(\mu,\nu)\,.
\end{align*}
As a consequence, there exists $C>0$ that depends only on $b$ and $d$ such that for all $\gamma\in (0,\gamma_0]$, 	$m, N\in\N$ and all $\mu,\nu\in\mathcal P(\T^{dN})$,
\begin{align*}
\mathcal W_{1}\po \mu R_{N,\gamma}^m,\nu R_{N,\gamma}^m\pf \ \leqslant \ CN \po 1-\gamma \kappa \pf^m\,.
\end{align*}
\end{thm}
This means that, with respect to the metric $\rho_N$, $R_{N,\gamma}$ has a Wasserstein curvature of $\gamma \kappa$ in the sense of \cite{Joulin}.

Theorem~\ref{Theorem-Inter} is proven in Section~\ref{SubSection-Tempslong}. From this first result, similar bounds can be obtained  for large $N$ and small $\gamma$ (see Sections~\ref{Subsec:propchaos} and \ref{Sec-continu}). Combining all these results eventually yields a quantitative bound on the error made in practice by approximating $\nu_*$ by the empirical distribution of the particular particle system:}
\begin{thm}\label{Theorem-Main}
\nvdd{Under the conditions of Theorem~\ref{Theorem-Inter}, suppose that $\kappa$   given by \eqref{Eq-ConditionPerturb} is positive. There exists $C>0$ such that for all $N\in\N$, $\gamma\in (0,\gamma_0]$, $t\geqslant 0$ and $\mu_0\in\mathcal P(\mathbb T^{dN})$, if $(\X_k)_{k\in\N}$ is a Markov chain with initial distribution $\mu_0$ and transition kernel $R_{N,\gamma}$, }
\[\mathbb E \co \mathcal W_1\po \pi( \X_{\lfloor t/\gamma\rfloor}), \nu_*\pf\cf \ \leqslant \ C \po  \sqrt \gamma + \alpha(N) + e^{-\kappa t}\pf\,,\]
where  
\[\alpha(N) \ = \ \left\{\begin{array}{ll}
N^{-1/2} & \text{if }d=1\,,\\
N^{-1/2}\ln(1+N) & \text{if }d=2\,,\\
N^{-1/d} & \text{if }d>2\,.
\end{array}\right.\]
\end{thm}

\nvdd{All the constants in Theorems~\ref{Theorem-Inter} and \ref{Theorem-Main} (and all other results stated in this work) are explicit. More precisely, $c_1$ and $\gamma_0$ come from    \cite[Corollary 2.2] {Majka} (see Proposition~\ref{Prop-ContractMajka} below) where an explicit value is given, and  all the other constants involved in our results can be tracked by following the explicit computations.}

In Theorem~\ref{Theorem-Main}, the speeds of the different convergences (exponential in the simulation time, with the square-root of the timestep and with $\alpha$ of the number of particles) are optimal since they are optimal for non-interacting diffusions (i.e. the case $\lambda = 0$), see in particular \cite{FournierGuillin} for the large $N$ asymptotic.

Other intermediary results will be established in the rest of the paper that are interesting by themselves: propagation of chaos (i.e. $N\rightarrow \infty$) and continuous-time limit at a fixed time (even without the condition $\kappa>0$) respectively in Propositions \ref{Prop-chaos} and \ref{Prop-continu}. From that, results for the continuous-time process ($\gamma=0$), the equilibria ($t=\infty$) or the non-linear process ($N=\infty$), or when two parameters among three are sent to their limits, are then simple corollaries, see Section~\ref{Sec-conclusion}. \nvdd{All these results are summarised in Figure~\ref{Figure_Corollaires}.}

Note that $\exp(-\gamma \lambda(x))$ is the probability that the chain is not killed when it arrives at state $x$. The time step $\gamma$ should be chosen in such a way that this probability is relatively large, say at least one half. In that case, $\exp(\gamma \|\lambda\|_\infty)$   is typically close to 1. In other words, the positivity of $\kappa$ given by \eqref{Eq-ConditionPerturb} is mostly a condition about $L_\lambda$ being small enough.

This perturbation condition is different from the one considered in \cite{Monmarche}, where $\|\lambda\|_\infty$ rather than $L_\lambda$ is supposed to be small (while our main arguments are a direct adaptation of the coupling arguments of \cite{Monmarche}). This difference comes from the fact that, in the present study, we work with the $\mathcal W_1$ distance rather than the total variation one (which is a Wasserstein distance but associated to the discrete metric $d(x,y)=\1_{x\neq y}$). Indeed, in our coupling arguments, we need to control $|\lambda(x) - \lambda(y)|$ the difference between the death rates of two processes at different locations, which is bounded here by  $L_\lambda |x-y|$ and in \cite{Monmarche} by $2\|\lambda\|_\infty\1_{x\neq y}$. In fact our argument for the long-time convergence may easily be adapted to the total variation distance framework, following \cite{Monmarche}. Nevertheless this would be more troublesome in the study of the limit $N\rightarrow\infty$. Then, one needs to couple $\eta_k$ (that admits a density with respect to the Lebesgue measure) with $\pi(\X_k)$ (which is a sum of Dirac masses), so that the total variation distance is not adapted. This may be solved by considering $\mathcal W_1\rightarrow $ total variation regularization results for (Euler schemes of) diffusions, that can be established by coupling arguments again. Nevertheless, in order to focus on the other difficulties of the problem and for the sake of clarity, we decided to stick to the $\mathcal W_1$  distance in all the different results of this work. \nvdd{Similar Wasserstein coupling arguments have been used in \cite{CloezThai} on a similar problem (see next section) and in \cite{Villemonais3} for a different kind of mean-field interacting particle system (also with a similar perturbative condition corresponding to the fact $\sigma$ in \cite[Proposition 3.1]{Villemonais3} has to be positive, i.e. the interaction should be small with respect to the independent mixing).}

\nvd{Notice that, among all possible discretization schemes, we only considered the explicit Euler-Maruyama one. This choice was made for simplicity, but the proofs could be extended to other usual schemes. The main ingredient required is   a Wasserstein curvature of order $\gamma$ for a modified $\mathcal W_1$ distance (see  Proposition~\ref{Prop-ContractMajka}, based on \cite[Corollary 2.2]{Majka}).} \nvdd{Similarly, we only considered the case of an elliptic diffusion process with a constant diffusion matrix for simplicity (since we use \cite[Corollary 2.2]{Majka} which covers this case), although a similar Wasserstein contraction certainly holds in a much more general framework (even hypoelliptic non-elliptic, as in  continuous-time settings \cite{EberleGuillinZimmer}). As stated in the introduction, the present paper does not aim at the broadest generality, and by avoiding technicalities  we want to highlight the main issue (i.e. the question of the uniformity of bounds in the various parameters).}

\subsection{Related works}\label{Sec:previous_works} 

\nvdd{
The use of a particle system with death and re-birth to approximate the QSD of a Markov process has been introduced in \cite{Burdzy1}, for two-dimensional Brownian motions killed at the boundary of a box. This work refers to the system as a \emph{Fleming Viot process}. However, in the lecture notes of Dawson \cite{Dawson}, a (continuous-time) system of $N$ particles that move independently according to some Markov dynamics and interact through a sampling-replacement mechanism is called a \emph{Moran particle process},  while the term \emph{Fleming-Viot process} refers to a measure-valued (continuous-time) process that can be obtained as the limit of the Moran particle system as the number of particles goes to infinity. Besides, with these definitions, the empirical measure of a Moran particle process is nothing but a Fleming-Viot process in the particular case where the initial condition is the sum of $N$ Dirac measures.  Both the initial works of Moran \cite{Moran} and Fleming and Viot  \cite{FlemingViot}  are motivated by population genetics models. 

The seminal work \cite{Burdzy1} is a numerical study so that, although a continuous-time continuous density-valued process is targeted, what is really implemented is in fact a discrete-time particle system. From then, the use of similar processes in numerical schemes (for killed processes or more general non-linear problems such as non-linear filtering, rare events analysis and so on \cite{DelMoralbouquin,DelMoralbouquin2}) have been widely studied. Although the term \emph{Moran particle system} is used in \cite{MicloDelMoral2} and a few other works,   most studies concerned with quasi-stationary distributions refer to \emph{Fleming-Viot particle system}, see e.g.  \cite{ferrari2007,GrigoKang,Guyadersoft,Guyaderhard,LelievreReygner,CloezThai,Lobus,CV2018,Asselah} (for continuous-time processes, which thus corresponds   to the process introduced in \ref{Sec-continu} i.e. the limit $\gamma\rightarrow 0$). 

Different frameworks have been considered: finite space, discrete infinite space, compact and non-compact continuous space; continuous and discrete time; hard or soft killing, or more general non-linear evolutions. There are also some variants on the precise definition of the rebirth mechanism (as mentioned in the introduction, our specific scheme, where killed particles at step $k$ are resurrected on a position  of a particle at step $k-1$ and then perform an Euler step conditioned not to be killed again, seems to be new). Disregarding these differences, let us discuss the kind of results established in the existing literature.

A first set of works, starting shortly after the initial numerical study of \cite{Burdzy1}, are concerned with finite-time propagation of chaos, either for the marginal laws or at the level of a trajectory in the Skorohod topology \cite{Burdzy2,GrigoKang,Villemonais,ferrari2007,MicloDelMoral2,Asselah}, possibly with a precise CLT \cite{Guyadersoft,Guyaderhard}. Similarly,  propagation of chaos and/or CLT as $N\rightarrow +\infty$ at stationarity (i.e. for the invariant measure of the Fleming-Viot particle system) are established e.g. in \cite{Asselah,LelievreReygner}. Uniform in time propagation of chaos is established in \cite{MicloDelMoral,DelMoralGuionnet,Rousset}. Contrary to our results, this uniformity in time is not obtained with a long-time convergence of the particle system at a speed independent from $N$, but rather from the long-time convergence of the limit ($N=+\infty$) non-linear process. This long-time convergence for the non-linear process (or, equivalently according to Proposition~\ref{Prop-EgaliteLoi}, for the process conditionned to survival) has recently been studied in general settings, in particular in a sery of work by Champagnat,  Villemonais and coauthors, see e.g.   \cite{ChampVilK2018,CV2020,DelMoralVillemonais2018,Bansaye} and references within. The idea to combine a finite-time convergence as $N\rightarrow +\infty$ with a long-time convergence of the limit process to obtain a uniform in time convergence with $N$ traces back at least to \cite{Norman}. Remark that, combining uniform in time propagation of chaos estimates with long-time convergence of the limit process, it is possible to obtain results in the spirit of Theorem~\ref{Theorem-Main}, i.e. that gives an error bound between the empirical measure of the chain simulated in practice with the target QSD, even if no long-time convergence of the particle system is available.

Contrary to the present paper,  most of  these previous works do not consider a perturbative framework where a condition similar to the positivity of $\kappa$ given by \eqref{Eq-ConditionPerturb} would be considered. Such a condition is considered in \cite{CloezThai}, which together with the present paper is the only one that establishes a long-time convergence rate uniform in $N$ for the particle system. Remark that   geometric ergodicity for the particle system, stated for instance in \cite{Villemonais2}, is usually easy to obtain (for a fixed $N$) from classical tools on Markov chains. Getting a rate that is uniform in time is a lot harder (hence the perturbative framework). Besides, for another class of mean-field particle systems, the McKean-Vlasov diffusions (for which interaction is induced by an interaction potential force in the drift of the diffusion), it is well-known that there are cases  (in non-convex confining potential with convex interaction for instance, at low temperature, for instance)  where the non-linear limit system has several equilibria and the convergence rate of the particle system (which has a unique invariant measure for all $N$) goes to $+\infty$ with $N$. In fact we don't expect this to happen for the Fleming-Viot particle system since, as mentioned above, the uniqueness of the QSD, the long-time convergence of the limit process and the uniform in time propagation of chaos have been established in non-perturbative cases.  \nvdd{This may indicate that the uniform in $N$ long-time convergence could hold in much more general cases, far from the perturbative regime around the non-interacting case. In that case, our perturbative condition (and the one of \cite{CloezThai}) would just come from the particular non-optimal proof. 
However, the long-time convergence of the limit-process does not imply directly the uniform convergence for the interacting system, so this is still an open question, and our results and those of \cite{CloezThai} are the only of their kind. At least, we can say that we have no reason to think that our condition $\kappa>0$, with the explicit expression of $\kappa$, is sharp in any way.} 

The differences of our work with \cite{CloezThai} are the following. The latter is concerned with a continuous-time Markov chain on a finite state space rather than a diffusion on the torus. Moreover, it requires a strong Doeblin condition: the parameter $\lambda$ in \cite{CloezThai}  is required to be positive, which implies that, for any pair of states $i,i'$,  there is a probability to go either from $i$ to $i'$, or from $i'$ to $i$, or from both $i$ and $i'$ to some third state $j$. 
 Related to this, in \cite{ferrari2007}, although uniform in $N$ long-time convergence is not stated or proven, a coupling argument similar to ours or to \cite{CloezThai} is used (in \cite[Proposition 3.1]{ferrari2007}) to obtain uniform in time propagation of chaos estimates. This work is concerned with countable state space under an even more restrictive Doeblin condition than \cite{CloezThai} (there exists at least one state $i$ for which the transition rate from $j$ to $i$ is uniformly bounded for all other $j$), and the uniform in time result requires a perturbative condition ($\alpha>C$ in \cite{ferrari2007} corresponding to $\lambda>0$ in \cite{CloezThai} and $\kappa>0$ for us). We remark that, in a countable discrete state space, our arguments can be easily adapted to obtain a uniform long-time convergence under a condition of positive Wasserstein curvature for some distance (similarly to Proposition~\ref{Prop-ContractMajka}), which is much less restrictive than the Doeblin conditions of \cite{ferrari2007,CloezThai}.

As far as the time discretization error is concerned, we are not aware of  results similar to ours in previous works but we refer the reader to \cite{FerreStoltz} for weak error studies \`a la Talay-Tubaro for some non-linear evolutions, and references within for more details.}

\nvd{Finally, let us mention another related set of works, \cite{BenaimCloezPanloup,BCV,benaim2015}, based on self-interacting processes. Indeed, the reason we introduced a system of $N$ interacting particles was to approximate some non-linearity in the evolution of a measure-valued process. Yet, actually, when it comes to the approximation of the QSD, we are not really interested in the non-linear evolution, but only in its long-time limit. A classical idea in the field of stochastic algorithms in that case is to construct a chain similar to $(Y_k)_{k\in\N}$ except that, in  the rebirth mechanism, the unknown law $\eta_n$ is replaced by the occupation measure of the past trajectory (which, by ergodicity, is expected to converge to equilibrium), i.e. there is only one particle and, when it dies at time $t>0$, it is resurrected at its position at time $s$ uniformly distributed over $[0,t]$. Although the algorithms are quite similar (and can be combined), their theoretical studies rely on quite distinct arguments.
}

%
%

\section{Proofs}\label{Sec-proofs}

Let us first establish the preliminary result stated in the introduction:

\begin{proof}[Proof of Proposition \ref{Prop-EgaliteLoi}]
For $n\in\N$, denote
\[\eta_n = \mathcal Law( Y_n)\,,\qquad \nu_n = \mathcal Law\po \tilde Z_n \ | \ \tilde T \geqslant n\gamma\pf\,.\]
Since $\nu_0=\eta_0$, suppose by induction that $\nu_n = \eta_n$  for some $n\in\mathbb N$.  Keeping the notations introduced of the definition of the kernel $Q_{\mu}$, consider the events $B_k = \{U_k \geqslant p(X_k)\}$. Then, for all bounded measurable $f$,
\begin{eqnarray*}
Q_\mu f(x) & = & \mathbb E\po f(X)\pf   \\
& = & \mathbb E\po f(X) \sum_{k\in\N} \1_{B_k \cap (\bigcap_{j=0}^{k-1} B_j^c)}\pf \\
& = & \mathbb E \po f(X_0)\1_{U_0 \geqslant p(X_0)}\pf + \sum_{k\geqslant 1} \mathbb E \po f(X_k)\1_{U_k \geqslant p(X_k)}\pf \prod_{j=0}^{k-1} \mathbb P \po B_j^c\pf \\
&= &  K\co f(1-p)\cf(x) + \sum_{k\geqslant 1}  \mu K\co f(1-p)\cf  \po \mu Kp\pf^{k-1} Kp(x) \,.
\end{eqnarray*}
In particular, integrating with respect to $\mu$, we obtain
\begin{eqnarray*}
\mu Q_\mu f & = & \mu K\co f(1-p)\cf \sum_{k\in\N} \po \mu Kp\pf^k\ = \ \frac{\mu K\co f(1-p)\cf }{\mu K\co 1-p\cf }\,.
\end{eqnarray*}
Applied with $\mu = \eta_n$, this reads
\begin{eqnarray*}
\eta_{n+1} f \ = \ \mathbb E\po f(Y_{n+1})\pf & = & \mathbb E\po \mathbb E\po f(Y_{n+1})\ | \ Y_n\pf\pf\ = \  \eta_n Q_{\eta_n} f\ =\  \frac{\eta_n K\co f(1-p)\cf }{\eta_n K\co 1-p\cf }\,.
\end{eqnarray*}
On the other hand, 
\begin{eqnarray*}
\mathbb E\po f(\tilde Z_{n+1})\1_{\tilde T>(n+1)\gamma}\pf & = & \mathbb E\po f\po \tilde Z_{n+1}\pf \1_{\tilde T>n\gamma} \1_{U_n \geqslant p(\tilde Z_{n+1})}\pf \\
& = & \mathbb E\po f\po \tilde Z_{n+1}\pf \po 1 - p\po \tilde Z_{n+1}\pf\pf \1_{\tilde T>n\gamma} \pf \\ 
& = & \mathbb P \po \tilde T>n\gamma\pf \nu_n K\co f(1-p)\cf\,,
\end{eqnarray*}
frow which
\[\nu_{n+1}f \ = \ \frac{\mathbb E\po f(\tilde Z_{n+1})\1_{\tilde T>(n+1)\gamma}\pf}{\mathbb P\po \tilde T>(n+1)\gamma \pf} \ = \ \frac{\mathbb P \po \tilde T>n\gamma\pf \nu_n K\co f(1-p)\cf}{\mathbb P \po \tilde T>n\gamma\pf \nu_n K\co 1-p\cf}  \ = \ \frac{ \nu_n K\co f(1-p)\cf}{  \nu_n K\co 1-p\cf}\,,\]
which concludes.
\end{proof}

\subsection{The basic coupling}\label{Subsec:basiccoupling}

The long-time estimates needed to prove convergence toward equilibrium and uniform in time estimates in $N$ and $\gamma$ are based on the fact that, as long as particles don't die, they follow the chain \eqref{Eq-EulerSchemeEDS} which, like its continuous-time counterpart \eqref{Eq-EDS}, have some mixing properties. In order to quantify the latters, we start by stating \cite[Corollary 2.2]{Majka} in a suitable way in our context.

\begin{prop}\label{Prop-ContractMajka}
There exists $c_1,a,\gamma_0>0$ (that all depend only on the drift $b$ of \eqref{Eq-EDS} and on the dimension $d$) such that, denoting $\rho(x,y) = (1-\exp(-a|x-y|))/a$ for $x,y\in\mathbb T^d$, then $\rho$ is a metric on $\mathbb T^d$ with
\[\forall \gamma \in (0,\gamma_0]\,, \ \forall \mu,\nu\in \mathcal P(\mathbb T^d)\,,\qquad \mathcal W_\rho\po \mu K,\nu K\pf \ \leqslant \ (1-c_1\gamma)\mathcal W_\rho(\mu,\nu)\,.\]
\end{prop}
\nvdd{
\begin{proof}
This is \cite[Corollary 2.2]{Majka} applied to a diffusion with smooth drift on the torus, in which case the distance on \nvd{$\mathbb T^d$} for which the contraction holds is $\rho$.  
\end{proof}}

In the rest of the paper, $\rho$ is the metric and $c_1,a,\gamma_0$ are the constants given by Proposition \ref{Prop-ContractMajka}. Remark that $\rho$ is equivalent to the Euclidian metric, with
\[\beta |x-y| \ \leqslant \ \rho(x,y) \ \leqslant \ |x-y|\qquad \text{for}\qquad \beta \ = \ 2(1-e^{-a\sqrt {d}/2})/(a\sqrt {d})\,,\]
where we used that the diameter of $\T^d$ is $\sqrt {d}/2$ and that $r\mapsto (1-\exp(-ar))/a$ is a concave function with derivative 1 at zero.  In particular, $\mathcal{W}_1$ and $\mathcal W_\rho$ are equivalent.


\bigskip

Now, in this particle system, the contraction property of the chain \eqref{Eq-EulerSchemeEDS} may be counterbalanced by the death/resurrection mechanism through which particles interact. Indeed, considering two systems of $N$ interacting particles, for $i\in\cco1,N\ccf$ the previous result means that we can couple the $i^{th}$ particles of both systems to get closer one to the other (on average), as long as they don't die. But then, one of the two particle can die and resurrect far from the other, or even if they die simultaneously they may resurrect far apart one from the other.  That being said, first, the closer they get, the easier it is to couple them in order to die simultaneously, and second, when they die simultaneously, keeping the particles close one to the other amount to do a suitable coupling of the laws from which the particles are resurrected. This is quantified in the following proposition.

In all the rest of the paper, we suppose that $\lambda$ is $L_\lambda$-Lipschitz (but not necessarily that $\kappa$ given by \eqref{Eq-ConditionPerturb} is positive).

\begin{prop}\label{Prop-CouplBasic}
Let \nvd{$\mu_0,\mu_1,\mu_0',\mu_1' \in\mathcal P(\T^d)$ and let $(X_0,X_0')$ (resp. $(X_1,X_1')$) be a coupling of  $\mu_0K$ and $\mu_0' K$ (resp. $\mu_1 K$ and $\mu_1' K$)}. Then 
\begin{eqnarray*}
\mathcal W_{\rho} \po \nvd{\mu_0 Q_{\mu_1},\mu_0' Q_{\mu_1'}}\pf & \leqslant & h \po \mathbb E \po \rho(X_0,X_0')\pf + \frac{q_0}{1-q_1}  \mathbb E \po\rho(X_1,X_1')\pf\pf
\end{eqnarray*}
where
\begin{eqnarray*}
h &=&   1 - \min p  + (a\beta)^{-1} \gamma L_\lambda
\end{eqnarray*}
and, considering $U\sim \mathcal U([0,1])$ independent from $(X_0,X_0',X_1,X_1')$,
\begin{eqnarray*}
q_i & = & \mathbb P \po U < p(X_i)\wedge p(X_i')\pf\,,\qquad i=0,1.
\end{eqnarray*}

\end{prop}

\begin{proof}
Let $(X_k,X_k',U_k)_{k\in\N}$ be a sequence of independent triplet of random variables such that, for all $k\in\N$, $U_k\sim \mathcal U([0,1])$ is independent from $(X_k,X_k')$, which are such as defined in the proposition for $k=0$ and $1$ and, for $j>1$, have the same distribution as $(X_1,X_1')$. Set $H=\inf\{n\in\N,\ U_n < p(X_n)\}$ and $H'=\inf\{n\in\N,\ U_n < p(X_n')\}$. Then, by considering the law of $(X_k,U_k)_{k\in\N}$ alone, it  is clear that $X_{H} \sim \mu_0 Q_{\mu_1}$ and, similarly, $X_{H'}' \sim \mu_0' Q_{\mu_1'}$, so that
\[\mathcal W_{\rho}\po \mu_0 Q_{\mu_1},\mu_0' Q_{\mu_1'}\pf \ \leqslant \ \mathbb E\po \rho\po X_{H},X'_{H'}\pf\pf\,.\]
Different cases are distinguished depending on the value of $H$ and $H'$. In the simplest case, none of the particles dies:
\begin{eqnarray*}
\mathbb E\po \rho\po X_{H},X_{H'}'\pf\1_{H=H'=0}\pf & = &  \mathbb E\po \rho\po X_{0},X'_{0}\pf\1_{U_0 \geqslant p(X_0)\vee p(X'_0)}\pf\\
& \leqslant & \mathbb E\po \rho\po X_{0},X'_{0}\pf\1_{U_0 \geqslant \min p}\pf\\
& \leqslant & \po 1 - \min p\pf \mathbb E\po \rho\po X_{0},X'_{0}\pf\pf\,,
\end{eqnarray*}
where we used 
 the independence between $U_0$ and $(X_0,X'_0)$. In the second case, only one particle dies: using that $\|\rho\|_\infty \leqslant 1/a$,
\begin{eqnarray*}
\mathbb E\po \rho\po X_{H},X'_{H'}\pf\1_{H\wedge H'=0<H\vee H'}\pf & \leqslant & a^{-1} \mathbb P \po U_0 \in [p(X_0)\wedge p(X'_0),p(X_0)\vee p(X'_0)]\pf\\
& = &  a^{-1}\mathbb E\po |p(X_0)-p(X'_0)|\pf\\
& \leqslant &  a^{-1}\gamma L_\lambda \mathbb E\po |X_0-X'_0|\pf\\
& \leqslant &  (a\beta)^{-1} \gamma L_\lambda\mathbb E\po \rho(X_0,X'_0)\pf\,.
\end{eqnarray*}
In the third case, both particles die $k\geqslant 1$ times: 
\begin{eqnarray*}
\mathbb E\po \rho\po X_{H},X'_{H'}\pf\1_{H=H'=k}\pf & = &  \mathbb E\po \rho\po X_{k},X'_{k}\pf\1_{U_k \geqslant p(X_k)\vee p(X'_k)}\prod_{j=0}^{k-1}\1_{U_j < p(X_j)\wedge p(X'_j)}\pf\\
& \leqslant & q_0 q_1^{k-1} \mathbb E\po \rho\po X_{k},X'_{k}\pf\1_{U_k \geqslant \min p}\pf\\
& \leqslant & q_0 q_1^{k-1} \po 1 - \min p\pf \mathbb E\po \rho\po X_{1},X'_{1}\pf\pf\,. 
\end{eqnarray*}
Finally, combining the computations of the last two cases, the fourth one reads, for $k\geqslant 1$,
\begin{eqnarray*}
\mathbb E\po \rho\po X_{H},X'_{H'}\pf\1_{H\wedge H'=k<H\vee H'}\pf & \leqslant & a^{-1} q_0 q_1^{k-1}  \mathbb P \po U_k \in [p(X_k)\wedge p(X'_k),p(X_k)\vee p(X'_k)]\pf\\
& \leqslant & (a\beta)^{-1} q_0 q_1^{k-1}   \gamma L_\lambda \mathbb E\po \rho\po X_{1},X'_{1}\pf\pf\,.  
\end{eqnarray*}
Summing these four cases concludes.
\end{proof}

\subsection{Long-time convergence}\label{SubSection-Tempslong}

 For $N\in\N_*$ denote $\rho_N$ the metric on $\mathbb T^{dN}$ given by
\[\rho_N(x,y) \ = \ \sum_{i=1}^N \rho(x_i,y_i)\,.\]

\nvdd{The following result is similar to the results of \cite{Monmarche,Villemonais3,CloezThai} and based on the same coupling argument.}

\begin{prop}\label{Prop-TempsLong}
There exists $c_2>0$ (that depends only on the drift $b$ of \eqref{Eq-EDS} and on the dimension $d$) such that for all $\gamma\in (0,\gamma_0]$ 	$N\in\N$, and all $\mu,\nu\in\mathcal P(\T^{dN})$,
\[ \mathcal W_{\rho_N}\po \mu R_{N,\gamma},\nu R_{N,\gamma}\pf \ \leqslant \ \po 1-\gamma \kappa \pf\mathcal W_{\rho_N}(\mu,\nu)\,.\]
with $\kappa$ given by \eqref{Eq-ConditionPerturb}.
\end{prop}
\begin{proof}
It is in fact sufficient to prove this for $\mu = \delta_x$ and $\nu = \delta_y$ for any $x,y\in\mathbb T^{dN}$. Indeed, assuming the result proven for Dirac masses, in the general case, considering $(\X_0,\Y_0)$ an optimal coupling of $\mu$ and $\nu$ and $(\X_1,\Y_1)$ an optimal coupling of $R(\X_0,\cdot)$ and $R(\Y_0,\cdot)$, then $\X_1 \sim \mu R$ and  $\Y_1\sim\nu R$, so that
\begin{eqnarray*}
W_{\rho_N}\po \mu R,\nu R\pf & \leqslant & \mathbb E \po \rho_N\po \X_1,\Y_1\pf\pf\\
& = & \mathbb E \po \mathbb E \po \rho_N\po \X_1,\Y_1\pf\ |\ (\X_0,\Y_0)\pf\pf\\
& = & \mathbb E \po  \mathbb E \po  \mathcal W_{\rho_N}(\delta_{\X_0} R,\delta_{\Y_0}R)\ |\ (\X_0,\Y_0)\pf\pf\\
& \leqslant & \po 1-\gamma\kappa\pf\mathbb E \po \rho_N\po \X_0,\Y_0\pf\pf \\
& = & \po 1-\gamma\kappa \pf \mathcal  W_{\rho_N}\po \mu ,\nu \pf\,.
\end{eqnarray*}
Hence, in the following, we fix $x,y\in\mathbb T^{dN}$. Let $(X_i,Y_i)_{i\in\cco 1,N\ccf}$ be independent pairs of random variables in $\T^d$ where, for all $i\in\cco 1,N\ccf$, $(X_i,Y_i)$ is an optimal coupling of $Q_{\pi(x)}(x_i,\cdot)$ and $Q_{\pi(y)}(y_i,\cdot)$. Then $(\X,\Y)$ is a coupling of $R(x,\cdot)$ and $R(y,\cdot)$, so that
\begin{eqnarray*}
\mathcal W_{\rho_N}\po \delta_x R,\delta_y R\pf \ \leqslant \ \mathbb E\po \rho_N\po \X,\Y\pf\pf &=& \sum_{i=1}^N \mathbb E\po \rho\po X_i,Y_i\pf\pf \\
& = & \sum_{i=1}^N \mathcal W_{\rho}\po Q_{\pi(x)}(x_i,\cdot),Q_{\pi(y)}(y_i,\cdot)\pf\,. 
\end{eqnarray*}
We want to apply Proposition \ref{Prop-CouplBasic} with $\mu_1 = \pi(x)$, $\mu_0 = \delta_{x_i}$, $\mu_1' = \pi(y)$ and $\mu_0'=\delta_{y_i}$. To do so, for all $i\in\cco 1,N\ccf$, we consider $(\tilde X_i,\tilde Y_i)$ an optimal coupling of $K(x_i,\cdot)$ and $K(y_i,\cdot)$. From Proposition \ref{Prop-ContractMajka}, 
\begin{eqnarray}\label{Eq-Tempslong-interm2}
\mathbb E\po \rho\po \tilde X_i,\tilde Y_i\pf\pf & \leqslant & \po 1 - c_1\gamma\pf \rho(x_i,y_i)\,.
\end{eqnarray}
Moreover, if $J\sim\mathcal U(\cco 1,N\ccf)$ is independent from the $(\tilde X_i,\tilde Y_i)$'s, we remark that $(\tilde X_J,\tilde Y_J)$ is a coupling of $\pi(x)K$ and $\pi(y)K$.  Proposition \ref{Prop-CouplBasic} applied with these couplings reads, for all $i\in\cco 1,N\ccf$,
\begin{eqnarray}\label{Eq-Tempslong-interm1}
\mathcal W_{\rho}\po Q_{\pi(x)}(x_i,\cdot),Q_{\pi(y)}(y_i,\cdot)\pf& \leqslant & h \po \mathbb E \po \rho(\tilde X_i,\tilde Y_i)\pf + \frac{q_i}{1-q_*}  \mathbb E \po\rho(\tilde X_J,\tilde Y_J)\pf\pf
\end{eqnarray}
where, if $U\sim\mathcal U([0,1])$ is independent from the previous variables,
\[ q_i \ := \ \mathbb P \po U < p(\tilde X_i)\wedge p(\tilde Y_i)\pf\]
and, conditionning on the value of $J$,
\[q_* \ := \  \mathbb P \po U < p(\tilde X_J)\wedge p(\tilde Y_J)\pf \ = \ \frac1N \sum_{i=1}^N q_i\,.\]
Summing \eqref{Eq-Tempslong-interm1} over $i\in\cco 1,N\ccf$ and applying \eqref{Eq-Tempslong-interm2} yields
\[\mathcal W_{\rho_N}\po \delta_x R,\delta_y R\pf \ \leqslant \   h \po \po 1 - c_1\gamma\pf \sum_{i=1}^N \rho(x_i,y_i) + \frac{N q_*}{1-q_*}  \mathbb E \po\rho(\tilde X_J,\tilde Y_J)\pf\pf\,.\]
Applying Proposition \ref{Prop-ContractMajka} again,
\[\mathbb E \po\rho(\tilde X_J,\tilde Y_J)\pf\ \ = \ \frac1N\sum_{i=1}^N \mathbb E \po\rho(\tilde X_i,\tilde Y_i)\pf \ \leqslant \ \frac1N \po 1 - c_1\gamma\pf \sum_{i=1}^N \rho(x_i,y_i) \,,\]
and the previous inequality becomes
\[\mathcal W_{\rho_N}\po \delta_x R,\delta_y R\pf \ \leqslant \   \frac{h   \po 1 - c_1\gamma\pf}{1-q_*} \rho_N(x,y) \,.\]
Bounding $1-q_* \geqslant 1-\max p \geqslant \exp(-\gamma\|\lambda\|_\infty)$  and $\max p - \min p \leqslant \sqrt {d/2}  \gamma L_\lambda$ yields
\begin{eqnarray*}
\frac{h   \po 1 - c_1\gamma\pf}{1-q_*}   & \leqslant & \po 1 - c_1\gamma\pf \frac{1 - \max p + \max p - \min p  + (a\beta)^{-1} \gamma L_\lambda }{1-\max p} \\
& \leqslant & 1 - c_1\gamma + \gamma L_\lambda e^{\gamma\|\lambda\|_\infty}  \po (a\beta)^{-1}  + \sqrt d\pf\,,
\end{eqnarray*}
which concludes.
\end{proof}
As a direct consequence,  Proposition~\ref{Prop-TempsLong} gives 
\[\forall m\in\N,\qquad \mathcal W_{\rho_N}\po \mu  R^{m},\nu R^{m}\pf \ \leqslant \ e^{-\kappa m \gamma}\mathcal W_{\rho_N}(\mu,\nu)\,,\]
with $\kappa$ that does not depends on $N$ nor  $\gamma$. Provided $\kappa>0$, and since $\mathcal P(\T^{dN})$ is complete for $\mathcal W_1$ (hence for $\mathcal W_{\rho_N}$) the Banach fixed-point theorem implies then that $R$ admits a unique invariant measure toward which it converges at rate $\gamma \kappa$.

In the rest of the paper, $\kappa$ is given by \eqref{Eq-ConditionPerturb} (but is not necessarily assumed positive).

\nvdd{
\begin{proof}[Proof of Theorem~\ref{Theorem-Inter}]
The first part of the theorem has already been proven in Proposition~\ref{Prop-TempsLong}. The last statement then follows from the first part, the
equivalence between the Euclidean distance and $\rho_N$ and the fact $\mathcal W_{\rho_N}(\nu,\mu)\leqslant \|\rho_N\|_{\infty} = N\sqrt{d}/2$ for all $\nu,\mu\in\mathcal P(\T^{dN})$.
\end{proof}

}
\subsection{Propagation of chaos}\label{Subsec:propchaos}

Recall that $\eta_k$ is the law at time $k$ of the non-homogeneous Markov chain $(Y_k)_{k\in\N}$ on $\T^d$ introduced in Section \ref{SubSec-DefProcess} with transition kernels $Q_{\eta_k}$ and initial condition $\eta_0$, and that $R=R_{N,\gamma}$ is the transition kernel of the Markov chain $(\X_k)_{k\in\N}$ on $\T^{dN}$.

\begin{lem}\label{Lem-Prop-chaos}
There exist $C_1>0$ such that for all $N\in\N$ , $\gamma\in(0,\gamma_0]$, $\eta\in\mathcal P(\T^d)$ and $\mu\in\mathcal P(\T^{dN})$,
\[\mathcal W_{\rho_N} \po \mu R,\mu Q_{\eta}^{\otimes N}\pf \ \leqslant \  \gamma  N  C_1 \int_{\T^{dN}} \mathcal W_\rho\po \pi(x),\eta\pf \mu(\dd x)\,.\]
\end{lem}
\begin{proof}
Similarly to the proof of Proposition \ref{Prop-TempsLong}, we start with the case  $\mu=\delta_x$ for some $x\in\T^{dN}$. Let $(X_i,Y_i)_{i\in\cco 1,N\ccf}$ be $N$ independent pairs of random variables such that for all $i\in\cco 1,N\ccf$, $(X_i,Y_i)$ is an optimal coupling of $Q_{\pi(x)}(x_i,\cdot)$ and $Q_{\eta}(x_i,\cdot)$. Then $(\X,\Y)$ is a coupling of $R_N(x,\cdot)$ and $Q_\eta^{\otimes N}(x,\cdot)$, so that
\begin{eqnarray*}
\mathcal W_{\rho_N} \po \delta_x R_{N},\delta_x Q_{\eta}^{\otimes N}\pf & \leqslant & \mathbb  E \po \rho_N\po X,Y\pf\pf\\
& = & \sum_{i=1}^N \mathbb E\po \rho(X_i,Y_i)\pf \ = \ \sum_{i=1}^N \mathcal W_{\rho}\po \delta_{x_i}Q_{\pi(x)},\delta_{x_i} Q_{\eta}\pf \,.
\end{eqnarray*}
From Proposition \ref{Prop-CouplBasic} (bounding $q_0\leqslant \max p \leqslant \gamma \|\lambda\|_\infty $ and $1-q_1 \geqslant 1 - \max p \geqslant \exp(-\gamma_0 \|\lambda\|_\infty)$)
\begin{eqnarray*}
\mathcal W_{\rho}\po \delta_{x_i}Q_{\pi(x)},\delta_{x_i} Q_{\eta}\pf  & \leqslant &    \gamma \|\lambda\|_\infty \po 1 + (a\beta)^{-1} \gamma_0 L_\lambda \pf e^{\gamma_0 \|\lambda\|_\infty}  \mathcal W_\rho\po \pi(x),\eta\pf \\
& := & \gamma C_1 \mathcal W_\rho\po \pi(x),\eta\pf \,.
\end{eqnarray*}
Now in the general case where $\mu$ is not a Dirac mass, considering $Z_0\sim \mu$, and $(Z_1,Z_2)$ an optimal coupling of $R(Z_0,\cdot)$ and  $Q_{\eta}^{\otimes N}(Z_0,\cdot)$ and conditioning with respect to $Z_0$, 
\[\mathcal W_{\rho_N} \po \mu R_{N},\mu Q_{\eta}^{\otimes N}\pf \ \leqslant \ \mathbb E\po \rho_N(Z_1,Z_N) \pf \ \leqslant \ \gamma N  C_1 \mathbb E \po \mathcal W_\rho\po \pi(Z_0),\eta\pf\pf\,. \]
\end{proof}

\begin{prop}\label{Prop-chaos}
 There exist $C_2,C_3>0$ such that for all $N\in\N$, $\gamma\in(0,\gamma_0]$, $m\in \N$ and $\eta_0\in\mathcal P(\T^d)$, first,
\begin{eqnarray}\label{Eq-PropChaos1}
\mathcal W_{\rho_N} \po \eta_0^{\otimes N} R^m,\eta_m^{\otimes N}\pf & \leqslant & C_2 N \alpha(N) \gamma \sum_{s=1}^m(1-\gamma\kappa)^{s-1}\,,
\end{eqnarray}
and second, if $(\X_k)_{k\in\N}$ is a Markov chain with initial distribution $\eta_0^{\otimes N}$ and transition kernel $R$, then
\begin{eqnarray}\label{Eq-PropChaos2}
\mathbb E\po \mathcal W_\rho \po \pi(\X_m),\eta_m\pf \pf & \leqslant & C_3 \alpha(N) \po 1 + \gamma \sum_{s=1}^m(1-\gamma\kappa)^{s-1}\pf\,.
\end{eqnarray}
\end{prop}

Remark that when $\kappa>0$, $\gamma \sum_{s=1}^m(1-\gamma\kappa)^{s-1} \leqslant 1/\kappa$ so that \eqref{Eq-PropChaos1}  and \eqref{Eq-PropChaos2} yield uniform in time estimates. On the contrary, when $k<0$, the estimates are exponentially bad in $t=m\gamma$.

\begin{proof}
We start with the proof of \eqref{Eq-PropChaos1}, for $m\geqslant 1$ (the case $m=0$ being trivial). From the triangular inequality, Proposition \ref{Prop-TempsLong} and Lemma \ref{Lem-Prop-chaos},
\begin{eqnarray*}
r_m \ :=\ \mathcal W_{\rho_N} \po \eta_0^{\otimes N}  R^m,\eta_m^{\otimes N}\pf  
& \leqslant & \mathcal W_{\rho_N} \po \eta_0^{\otimes N}  R^m,\eta_{m-1}^{\otimes N} R\pf  + \mathcal W_{\rho_N} \po \eta_{m-1}^{\otimes N} R, \eta_{m-1}^{\otimes N} Q_{\eta_{m-1}}^{\otimes N} \pf \\
& \leqslant & (1-\kappa\gamma) r_{m-1} + \gamma N C_1 \int_{\T^{dN}} \mathcal W_\rho\po \pi(x),\eta_{m-1}\pf \eta_{m-1}^{\otimes N}(\dd x)\,.
\end{eqnarray*}
Since $\mathcal W_\rho \leqslant \mathcal W_1$, estimating the last term is a classical question, that is to bound the expected Wasserstein distance between the empirical measure of a sample of $N$ independent and identically distributed random variables and their common law. From \cite[Theorem 1]{FournierGuillin} (and since on the torus the moments of probability measures are uniformly bounded), there exists some $C'>0$ independent from $\eta_0$, $m$, $N$ and $\gamma$ such that
\[ \int_{\T^{dN}} \mathcal W_1\po \pi(x),\eta_{m-1}\pf \eta_{m-1}^{\otimes N}(\dd x) \ \leqslant \ C'\alpha(N)\,.\]
Since $r_0 = 0$, a direct induction  concludes the proof of \eqref{Eq-PropChaos1}. 
%

 To prove \eqref{Eq-PropChaos2}, let $(\X,\Y)$ be an optimal coupling of $\eta_0^{\otimes N} R^m$ and $\eta_m^{\otimes N}$. Considering $J\sim \mathcal U(\cco 1,N\ccf)$ independent from $(\X,\Y)$ then, conditionally to $(\X,\Y)$, $(X_J,Y_J)$ is a coupling of $\pi(\X)$ and $\pi(\Y)$, so that
\[\mathcal W_\rho \po \pi (\X),\pi(\Y) \pf \ \leqslant \ \mathbb E\po \rho(X_J,Y_J)\ | \ (\X,\Y)\pf \ = \ \frac1N\rho_N(\X,\Y)\,.\] 
Taking the expectation in
\[\mathcal W_\rho \po \pi(\X),\eta_k\pf \ \leqslant \ \mathcal W_\rho \po \pi(\X),\pi(\Y)\pf + \mathcal W_\rho \po \pi(\Y),\eta_k\pf\,,\]
we conclude with \eqref{Eq-PropChaos1} and  \cite[Theorem 1]{FournierGuillin} again.
\end{proof}

\begin{cor}\label{Cor-Chaos}
With the notations of Proposition \ref{Prop-chaos}, for all $k\in\cco 1,N\ccf$,
\begin{eqnarray*}
\mathcal W_{\rho_k} \po \mathcal Law(X_{1,m},\dots,X_{k,m}),\eta_m^{\otimes k}\pf & \leqslant & C_2 k \alpha(N) \gamma \sum_{s=1}^m(1-\gamma\kappa)^{s-1}\,.
\end{eqnarray*}
\end{cor}
\begin{proof}
Let $(\X,\Y)$ be an optimal coupling of $\eta_0^{\otimes N} R^m$ and $\eta_m^{\otimes N}$, and let $\sigma$ be uniformly distributed over the set of permutations of $N$ elements,  independent from $(\X,\Y)$. Since the laws of $\X$ and $\Y$ are exchangeable, $\X_\sigma = (X_{\sigma(1)},\dots,X_{\sigma(N)})$ has the same law as $\X$, in particular $(\X_{\sigma(1)},\dots,X_{\sigma(k)})$ has the same law as $(X_1,\dots,X_k)$. The same goes for $\Y_\sigma$, and
\begin{eqnarray*}
\mathbb E\po \sum_{i=1}^k \rho(X_{\sigma(i)},Y_{\sigma(i)})\pf \ = \ k\mathbb E\po  \rho(X_{\sigma(1)},Y_{\sigma(1)})\pf & = & \frac kN \mathbb E \po \rho_N(\X,\Y)\pf \\
&=& \frac k N \mathcal W_{\rho_N} \po \eta_0^{\otimes N} R^m,\eta_m^{\otimes N}\pf\,,
\end{eqnarray*}
and Proposition \ref{Prop-chaos} concludes.
\end{proof}
Corollary \ref{Cor-Chaos} means that, for any fixed $k\in\N_*$, as $N$ goes to infinity, the $k$-marginals of the system of particles converge toward the law of $k$ independent non-linear chains, which is the so-called propagation of chaos phenomenon.

\subsection{Discrete to continuous time}\label{Sec-continu}

We start by defining $ (\overline{Y}_t)_{t\geqslant 0}$ and $(\bX_t)_{t\geqslant 0}$ the continuous-time analoguous of the chains $(Y_k)_{k\in\N}$ on $\T^d$ and  $(\X_k)_{k\in\N}$ on $\T^{dN}$ defined in Section \ref{SubSec-DefProcess}. We start with the non-linear process. For $t\geqslant 0$, let
\[\overline{\eta}_t \ = \ \mathcal Law(Z_t\ |\ T>t)\]
where $Z$ solves \eqref{Eq-EDS} with initial distribution $\eta_0$ and $T$ is given by \eqref{Eq-DeathTime}. We define $(\overline Y_t)_{t\geqslant 0}$ as follows. Set $\overline Y_0=Z_0\sim\eta_0$, $T_0=0$ and suppose that $T_n$ and $(\overline Y_t)_{t\in[0,T_n]}$ have been defined for some $n\in\N$. Let $(B_t)_{t\geqslant0}$ be a new Brownian motion on $\T^d$ and $E\sim \mathcal E(1)$, independent one from the other. Let $\tilde Y$ be the solution of 
\[\dd \tilde Y_t \ = \ b(\tilde Y_t)\dd t + \dd B_t\]
for $t\geqslant T_n$  with $\tilde Y_{T_n} = \overline Y_{T_n}$ and let 
\[T_{n+1} \ =  \  \inf\left\{t>T_n,\ E\leqslant \int_{T_n}^t \lambda (\tilde Y_s)\dd s\right\}\,.\]
For $t\in(T_n,T_{n+1})$, set $\overline Y_t = \tilde Y_t$. Finally, draw  a new $\overline Y_{T_{n+1}}$ according to $\overline{\eta}_{T_{n+1}}$. By induction $T_n$ and $(Y_t)_{t\in[0,T_n]}$ are then defined for all $n\in\N$. Since $\lambda$ is bounded, $T_n$ almost surely goes to infinity when $n\rightarrow \infty$ so that $(\overline Y_t)_{t\geqslant 0}$ is defined for all $t\geqslant 0$. Similarly to Proposition \ref{Prop-EgaliteLoi}, it can be established that $\mathcal Law(\overline Y_t) = \overline \eta_t$ for all $t\geqslant 0$.

Now, as in Section \ref{SubSec-DefProcess}, from the non-linear process $ (\overline{Y}_t)_{t\geqslant 0}$, the interacting particles $(\bX_t)_{t\geqslant 0}$ are obtained by replacing $\overline{\eta}_t$ by the empirical distribution of the system when particles die and are resurrected.

More precisely, let $(E_{i,k},B_{i,k},J_{i,k})_{i\in\cco 1,N\ccf,k\in\N}$ be a family of independent triplet of independent random variables where, for all $i\in\cco 1,N\ccf$ and $k\in\N$, $E_{i,k}\sim \mathcal E(1)$, $J_{i,k}\sim \mathcal U(\cco 1,N\ccf)$ (except if $k=0$, in which case $J_{i,k}=i$ almost surely) and $B_{i,k}=(B_{i,k,t})_{t\geqslant 0}$ is a $d$-dimensional Brownian motion. From these variables, we simultaneously define by induction the process and its death times $(T_{i,k})_{i\in\cco 1,N\ccf,k\in\N}$ as follows. First, set $\bX_0=x$ and $T_{i,0} = 0$ for all $i\in\cco 1,N\ccf$. For all $i\in\cco 1,N\ccf$, set $\hat X_{i,0,0}=x_i$ and for $k\geqslant 1$, set 
\begin{eqnarray}\label{Eq-DefContinuResurrect}
 \hat X_{i,k,T_{i,k}} &= &\lim_{t\overset{<}\rightarrow T_{i,k}} \overline{X}_{J_{i,k},t}\,.
\end{eqnarray}
For all $k\in\N$, for $t\geqslant T_{i,k}$, let $\hat  X_{i,k}$ solve 
\[\dd \hat  X_{i,k,t} \ = \ b\po \hat X_{i,k,t}\pf \dd t + \dd B_{i,k,t}\,,\]
 set
\[T_{i,k+1} \ = \ T_{i,k} + \inf\left \{t\geqslant 0,\ E_{i,k} \leqslant \int_0^t \lambda \po \hat  X_{i,k,s}  \pf \dd s \right\}\]
and for all $t\in [T_{i,k},T_{i,k+1})$, set $\overline{X}_{i,t} = \hat  X_{i,k,t}$.

Then $\bX_t = (\overline{X}_{1,t},\dots,\overline{X}_{N,t})$ is well-defined for all $t\geqslant 0$. Indeed, it is well defined for all $t< S_1 := \min\{ T_{i,1},\ i\in\cco 1,N\ccf\}$ the first death time of some particle, and is equal on this interval to $(\hat X_{1,0,t},\dots,\hat X_{N,0,t})$, which is continuous on $[0,S_1]$. Hence, the limits involved in \eqref{Eq-DefContinuResurrect} are well defined for $k=1$ and all $i\in\cco 1,N\ccf$ such that $T_{i,1}=S_1$. Then the algorithm above similarly defines the process up to the second time some particles die, etc. 

Remark that most of the times \eqref{Eq-DefContinuResurrect}  simply reads $\hat X_{i,k,T_{i,k}} = \overline{X}_{J_{i,k},T_{i,k}}$ (at its $k^{th}$ death time, the $i^{th}$ particle is resurrected at the current position of the $J_{i,k}^{th}$ particle). Indeed, the only case when this is not true is when the $J_{i,k}^{th}$ particle dies at time $T_{i,k}$. Since the   probability that two or more particles die simultaneously is zero, this almost surely only occurs if $J_{i,k} = i$, i.e. if the particle is resurrected at its own position.  

Denote $(P_{t})_{t\geqslant 0}$ the Markov semi-group associated with $(\bX_t)_{t\geqslant 0}$, i.e. for all $t\geqslant 0$, $P_{t}$ is the Markov kernel given by
\[P_{t} f(x) \ = \ \mathbb E\po f(\X_t)\ | \ \X_0 = x\pf\,.\]
We sometimes write $P_t = P_{N,t}$ to specify the number of particles.

\begin{lem}\label{Lem-continu}
There exist $C_4>0$ such that for all $N\in\N$, $\gamma\in(0,\gamma_0]$ and $\mu\in\mathcal P(\T^{dN})$,
\[\mathcal W_{\rho_N} \po \mu R_{N,\gamma},\mu P_{N,\gamma}\pf \ \leqslant \  N C_4\gamma^{3/2}\,.\]
\end{lem}

\begin{proof}
As in the proof of Lemma~\ref{Lem-Prop-chaos}, it is sufficient to treat the case $\mu=\delta_x$ with a fixed $x\in\T^{dN}$. Let $(\bX_t)_{t\geqslant0}$ be defined as above from random variables $(E_{i,k},B_{i,k},J_{i,k})_{i\in\cco 1,N\ccf,k\in\N}$. In particular, $\bX_{\gamma} \sim \delta_x P_{\gamma}$.

To define $\X_1 \sim \delta_x R$, for all $i\in\cco 1,N\ccf$ and $k\in \N$, consider $(\tilde X_{i,k,t})_{t\geqslant 0}$ the solution to $\tilde X_{i,k,0}=x_{J_{i,k}}$ and
\[\dd \tilde  X_{i,k,t} \ = \ b\po \tilde X_{i,k,0}\pf \dd t + \dd B_{i,k,t}\,.\]
Denoting 
\[H_i = \inf\left \{k\in\N,\ E_{i,k} \geqslant \gamma \lambda\po \tilde  X_{i,k,\gamma}\pf\right\}\,,\]
set $\X_1 := (\tilde X_{1,H_1,\gamma},\dots,\tilde X_{N,H_N,\gamma})$.

Then $(\X_1,\bX_\gamma)$ is a coupling of $R(x,\cdot)$ and $P_{\gamma}(x,\cdot)$, so that
\[\mathcal W_{\rho_N}\po R(x,\cdot),P_{\gamma}(x,\cdot)\pf \ \leqslant \ \mathbb E\po \rho_N( \X_1,\bX_\gamma)\pf \ = \ \sum_{i=1}^N \mathbb E\po \rho( X_{i,1},\overline{X}_{i,\gamma})\pf \,.  \]
 We now distinguish four cases, considering the events 
\begin{eqnarray*}
B_{i,1} &=& \{H_i = 0\text{ and }T_{i,1} >\gamma\}\\
B_{i,2} &=& \{H_i = 1\text{ and }T_{i,1} \leqslant \gamma < T_{i,2}\wedge T_{J_{i,0},1}\}\\
B_{i,3}& =& \{H_i = 1\text{ and }T_{i,1} > \gamma  \}\cup \{H_i = 0\text{ and }T_{i,1} \leqslant \gamma \}\\
B_{i,4} &=& \{H_i\geqslant 2\}\cup\{ T_{i,2}\leqslant \gamma\}\cup\{ T_{i,1}\vee T_{J_{i,0},1}\leqslant \gamma\}\,,
\end{eqnarray*}
that is, respectively: none of the two $i^{th}$ particles dies; both the $i^{th}$ particles die exactly once; one particle dies but not the other; at least two deaths are involved for one of the two particle. For all $i\in\cco 1,N\ccf$, $\Omega \ = \ \cup_{j=1}^4 B_{i,j}$, so that
\[\mathbb E\po \rho( X_{i,1},\overline{X}_{i,\gamma})\pf \ \leqslant \ \mathbb E\po \rho( X_{i,1},\overline{X}_{i,\gamma}) \po \1_{B_{i,1}} +\1_{B_{i,2}} +\1_{B_{i,3}} +\1_{B_{i,4}} \pf \pf \,.  \]
Conclusion follows by gathering the four cases.

\bigskip

\noindent \textbf{Case 1.} It reduces to the classical case of diffusions, since
\begin{eqnarray*}
\mathbb E\po |X_{i,1}-\overline{X}_{i,\gamma}| \1_{B_{i,1}} \pf  & = & \mathbb E\po | \tilde X_{i,0,\gamma}-\hat X_{i,0,\gamma}| \1_{B_{i,1}} \pf\  \leqslant \ \mathbb E\po |\tilde X_{i,0,\gamma}-\hat X_{i,0,\gamma}|\pf\,.
\end{eqnarray*}
Then
\begin{eqnarray*}
 |\tilde X_{i,0,t}-\hat X_{i,0,t}| & = & \left|  \int_0^t \po b(x_i)-b\po \hat X_{i,0,s}\pf\pf \right| \dd s\\
 & \leqslant &    \|\nabla b\|_\infty \int_0^t \po |\tilde X_{i,0,s}-\hat X_{i,0,s}| + |x_i-\tilde X_{i,0,s}|\pf \dd s 
\end{eqnarray*}
By the Gronwall Lemma, for all $t\geqslant 0$, almost surely,
\begin{eqnarray}\label{Eq-DemoCoupleGamma1}
\underset{s\in[0,t]}\sup |\tilde X_{i,0,t}-\hat X_{i,0,t}|  & \leqslant &    \|\nabla b\|_\infty e^{t \|\nabla b\|_\infty } \int_0^t |x_i-\tilde X_{i,0,s}| \dd s \,.
\end{eqnarray}
Since $\tilde X_{i,0,s}$ is a Gaussian variable with mean $x_i+sb(x_i)$ and variance $s$,
\begin{eqnarray}\label{Eq-DemoCoupleGamma}
\mathbb E\po |x_i-\tilde  X_{i,0,s}|  \pf & \leqslant &  s b(x_i) +  \mathbb E\po |x_i+sb(x_i)-\tilde  X_{i,0,s}|  \pf \ \leqslant \ \| b\|_\infty s + \sqrt s  \,. 
\end{eqnarray}
As a consequence, for $\gamma \leqslant \gamma_0$,
\begin{eqnarray}\label{Eq-EulerEstimate}
\mathbb E\po |\tilde X_{i,0,\gamma}-\hat X_{i,0,\gamma}|\pf & \leqslant &   \| \nabla b\|_\infty e^{\gamma_0\|\nabla b\|_\infty} \int_0^\gamma  \mathbb E\po |x_i-\tilde X_{i,0,s}|\pf \dd s \ \leqslant \     c \gamma^{3/2} \,.
\end{eqnarray}
\bigskip

\noindent \textbf{Case 2.} We bound
\begin{eqnarray*}
\mathbb E\po |X_{i,1}-\overline{X}_{i,\gamma}| \1_{B_{i,2}} \pf  & \leqslant & \mathbb E\po \po | \tilde X_{i,1,\gamma}-x_{J_{i,0}}|  + | \hat X_{i,1,\gamma}- x_{J_{i,0}}|\pf \1_{B_{i,2}} \pf\,.
\end{eqnarray*}
Similarly to \eqref{Eq-DemoCoupleGamma},
\begin{eqnarray*}
\mathbb E\po  | \tilde X_{i,1,\gamma}-x_{J_{i,0}}|  \1_{B_{i,2}} \pf &\leqslant & \mathbb E\po  | \tilde X_{i,1,\gamma}-x_{J_{i,0}}|  \1_{E_{i,0} \leqslant \gamma \|\lambda\|_\infty} \pf   \ \leqslant \ c \gamma^{3/2}   \,,
\end{eqnarray*}
where we used the independence of $E_{i,0}$ from $J_{i,1}$ and $(\tilde X_{i,1,t})_{t\geqslant 0}$. Denote $(X'_i)_{t\geqslant 0}$ the solution of
\[\dd    X_{i,t}' \ = \ b\po  X_{i,t}'\pf \dd t + \left\{\begin{array}{ll}
\dd B_{J_{i,0},0,t} & \text{for } t< T_{i,0} \\
\dd B_{i,1,t} & \text{for }t\geqslant T_{i,0}\,. 
\end{array}\right.\]
with $ X_{i,0}' = x_{J_{i,0}}$. Under the event $B_{i,2}$, $\hat X_{i,1,\gamma} = X_{i,\gamma}'$. Moreover, $J_{i,0}$, $B_{J_{i,0},0}$ and  $B_{i,1}$ are independent from $T_{i,0}$ and thus, by the strong Markov property, $(X_{i,t}')_{t\geqslant 0}$ is independent from $T_{i,0}$ and conditionally to $J_{i,0}$ it  has the same distribution as $\hat X_{J_{i,0},0,t}$ (namely it is a diffusion solving \eqref{Eq-EDS} with initial condition $x_{J_{i,0}}$). Hence,
\begin{eqnarray*}
\mathbb E  \po | \hat X_{i,1,\gamma}- x_{J_{i,0}}| \1_{B_{i,2}} \pf & \leqslant  & \mathbb E  \po | X_{i,\gamma}'- y_{J_{i,0}}| \1_{E_{i,0} \leqslant \gamma \|\lambda\|_\infty} \pf \ \leqslant c' \gamma ^{3/2}\,.
\end{eqnarray*}

\bigskip

\noindent \textbf{Case 3.} We bound 
\begin{multline*}
\mathbb E\po \rho(X_{i,1},\overline{X}_{i,\gamma}) \1_{B_{i,3}} \pf   \ \leqslant \ \frac1a \mathbb P\po B_{i,3} \pf \\
\ \leqslant \ \frac1a \mathbb P\po \int_0^\gamma  \lambda(\hat X_{i,0,s})\dd s \wedge \po \gamma \lambda(\tilde X_{i,0,\gamma})\pf \leqslant E_{i,0} \leqslant \int_0^\gamma \lambda(\hat X_{i,0,s})\dd s \vee \po \gamma \lambda(\tilde X_{i,0,\gamma})\pf \pf\\
\ =  \ \frac1a\mathbb E\po \left|\exp\po - \int_0^\gamma  \lambda(\hat X_{i,0,s})\dd s\pf - \exp\po - \gamma \lambda(\tilde X_{i,0,\gamma})\pf\right|\pf\\
\ \leqslant \ \frac1a\mathbb E\po \left| \int_0^\gamma  \lambda(\hat X_{i,0,s})\dd s  - \gamma \lambda(\tilde X_{i,0,\gamma})\right|\pf \,.
\end{multline*}
Now,
\begin{eqnarray*}
\left| \int_0^\gamma  \lambda(\hat X_{i,0,s})\dd s  - \gamma \lambda(\tilde X_{i,0,\gamma})\right| & \leqslant & L_\lambda \po   \int_0^\gamma  |\hat X_{i,0,s} - x_i  | \dd s+  \gamma  |  x_i   -   \tilde X_{i,0,\gamma}  |\pf\,.
\end{eqnarray*}
Using \eqref{Eq-DemoCoupleGamma1} together with \eqref{Eq-DemoCoupleGamma} yields
\[\mathbb E\po \rho\po X_{i,1},\overline{X}_{i,\gamma}\pf \1_{B_{i,3}} \pf   \ \leqslant \ c_3 \gamma^{3/2}\,. \]

\bigskip

\noindent \textbf{Case 4.} We bound 
\begin{eqnarray*}
\mathbb E\po \rho\po X_{i,1},\overline{X}_{i,\gamma}\pf \1_{B_{i,4}} \pf   & \leqslant & \frac1a \mathbb P\po B_{i,4} \pf \\
& \leqslant & \mathbb P \po E_{i,0}\vee E_{i,1} \leqslant \gamma \|\lambda\|_\infty\pf + \mathbb P \po E_{i,0}\vee E_{J_{i,0},0} \leqslant \gamma \|\lambda\|_\infty\pf\\
& \leqslant & 2 \po 1 - e^{-\gamma \|\lambda\|_\infty}\pf ^2 \ \leqslant\ 2\gamma^2 \|\lambda\|_\infty^2\,.
\end{eqnarray*}

\end{proof}

\begin{prop}\label{Prop-continu}
There exist $C_5>0$ such that for all $N\in\N$ , $\gamma\in(0,\gamma_0]$ and $\eta_0\in\mathcal P(\T^{d})$,
\begin{eqnarray*}
\mathcal W_{\rho_N} \po \mu R^m_{N,\gamma},\mu P_{N,m\gamma}\pf & \leqslant & \sqrt\gamma N C_5  \gamma \sum_{s=1}^m(1-\gamma\kappa)^{s-1}\,.
\end{eqnarray*}
\end{prop}

As for Proposition \ref{Prop-chaos}, when $\kappa>0$, $\gamma \sum_{s=1}^m(1-\gamma\kappa)^{s-1} \leqslant 1/\kappa$ so that \eqref{Eq-PropChaos1}  and \eqref{Eq-PropChaos2} yield uniform in time estimates. On the contrary, when $\kappa<0$, the estimates are exponentially bad in $t=m\gamma$.

\begin{proof}
The proof is similar to Proposition \ref{Prop-chaos}. Denoting $\mu_m = \mu R^m$ and $\nu_m = \mu P_{m\gamma}$, from the triangular inequality, Proposition \ref{Prop-TempsLong} and Lemma \ref{Lem-continu},
\begin{eqnarray*}
r_m \ :=\ \mathcal W_{\rho_N} \po \mu_m,\nu_m\pf & \leqslant & \mathcal W_{\rho_N} \po \mu_m,\nu_{m-1} R\pf  + \mathcal W_{\rho_N} \po \nu_{m-1} R,\nu_{m-1} P_{\gamma}\pf \\
& \leqslant & (1-\gamma\kappa) r_{m-1} + N C_4\gamma^{3/2}\,,
\end{eqnarray*}
and an induction concludes.
\end{proof}

\subsection{Conclusion}\label{Sec-conclusion}

In this section we use the notations of the previous ones, in particular $\kappa$ is given by \eqref{Eq-ConditionPerturb} and the constants $C_2$, $C_3$ and $C_5$  are those of Propositions \ref{Prop-chaos} and \ref{Prop-continu}. We can now gather all these previous results.

\nv{
Letting either $\gamma$ vanish or $N$ go to infinity in Proposition \ref{Prop-TempsLong}, we obtain long-time convergence for, respectively, the non-homogeneous self-interacting Markov chain $(Y_k)_{k\in\N}$ introduced in Section \ref{SubSec-DefProcess} and the continuous-time Markov chain $(\bX_t)_{t\geqslant 0}$ defined in Section \ref{Sec-continu}.

\begin{cor}\label{Cor-nonlineairetempslong}
Let $(\eta_n)_{n\in\N}$ be such as defined in Section \ref{SubSec-DefProcess}, and $(\tilde \eta_n)_{n\in\N}$ be similarly defined but with a different initial distribution $\tilde \eta_0 \in \mathcal P(\T^d) $. For all $m\in\N$ and all $\gamma\in(0,\gamma_0]$, 
\[ \mathcal W_{\rho}\po \eta_m,\tilde \eta_m\pf \ \leqslant \ \po 1-\gamma \kappa \pf^m \mathcal W_{\rho}(\eta_0,\tilde \eta_0)\,.\]
\end{cor}

\begin{cor}\label{Cor-tempslong-continu}
For all $N\in\N_*$, $t\geqslant 0$ and $\mu,\nu\in\mathcal P(\T^{dN})$,
\[\mathcal W_{\rho_N}\po \mu P_{N,t},\nu P_{N,t}\pf \ \leqslant \ e^{-\kappa t} \mathcal W_{\rho_N}\po \mu  ,\nu \pf \,.\]
\end{cor} 

\begin{proof}[Proof of Corollary \ref{Cor-nonlineairetempslong}]
The proof is based on the simple equality:
For all $N\in\N$ and $\mu,\nu\in\mathcal P(\T^d)$,
\begin{equation}\label{lemme}
\mathcal W_{\rho_N} \po \mu^{\otimes N},\nu^{\otimes N}\pf \ = \ N \mathcal W_{\rho}\po \mu,\nu\pf\,. 
\end{equation}
Indeed, by considering $N$ independent couplings $(X_i,Y_i)_{i\in\cco 1,N\ccf}$,
\[\mathcal W_{\rho_N}\po \mu^{\otimes N},\nu^{\otimes N}\pf \ \leqslant\ \mathbb E\po \rho_N(\X,\Y)\pf \ = \ \sum_{i=1}^N \mathbb E\po \rho(X_i,Y_i)\pf \ = \ N \mathcal W_{\rho}\po \mu,\nu\pf\,.  \]
Conversely, if $(\X,\Y)$ is an optimal coupling of $\mu^{\otimes N}$ and $\nu^{\otimes N}$, then
\[\mathcal W_{\rho}\po \mu,\nu\pf \ \leqslant \ \mathbb E\po \rho (X_1,Y_1)\pf \ = \ \frac1N \mathbb E\po \rho_N(\X,\Y)\pf \ = \ \frac1N \mathcal W_{\rho}\po \mu^{\otimes N},\nu^{\otimes N}\pf \,.\]
By the triangular inequality,
\begin{eqnarray*}
\mathcal W_{\rho}\po \eta_m^{\otimes N},\tilde \eta_m^{\otimes N}\pf & \leqslant & \mathcal W_{\rho}\po \eta_m^{\otimes N},\eta_0^{\otimes N} R^m\pf + \mathcal W_{\rho}\po \eta_0^{\otimes N} R^m,\tilde \eta_0^{\otimes N} R^m\pf + \mathcal W_{\rho}\po \tilde \eta_0^{\otimes N} R^m,\tilde \eta_m^{\otimes N}\pf\\
& \leqslant &   \po 1-\gamma \kappa \pf^m \mathcal W_{\rho}(\eta_0^{\otimes N},\tilde \eta_0^{\otimes N}) + 2 C_2 N \alpha(N) \gamma \sum_{s=1}^m(1-\gamma\kappa)^{s-1}\,.
\end{eqnarray*}
where we applied Propositions \ref{Prop-TempsLong} and \ref{Prop-chaos}. Using the equality \ref{lemme}, dividing by $N$ and letting $N$ go to infinity concludes the proof of Corollary \ref{Cor-nonlineairetempslong}. 
\end{proof}
Remark that the beginning of the proof also applies for $\mu,\nu \in \mathcal P(\T^{dN})$ that are exchangeable (i.e. invariant by any permutation of the $d$-dimensional coordinates), in which case, denoting, $\mu^{(1)}$ and $\nu^{(1)}$ their $d$-dimensional marginals, we get that
\[\mathcal W_{\rho} \po \mu^{(1)},\nu^{(1)}\pf \ \leqslant \ \frac1N \mathcal W_{\rho_N}\po\mu,\nu\pf\,.\]
\begin{proof}[Proof of Corollary \ref{Cor-tempslong-continu}]
Similarly to the previous proof, Corollary \ref{Cor-tempslong-continu} is a direct consequence of Propositions \ref{Prop-TempsLong} and \ref{Prop-continu}, letting $m$ go to infinity at a fixed $t$ and $N$ in
\begin{multline*}
\mathcal W_{\rho_N}\po \mu P_{N,t},\nu P_{N,t}\pf \ \leqslant \ \mathcal W_{\rho_N}\po \mu P_{N,t},\mu R_{N,t/m}^m\pf + \mathcal W_{\rho_N}\po \mu R_{N,t/m}^m,\nu R_{N,t/m}^m\pf\\ + \mathcal W_{\rho_N}\po \nu R_{N,t/m}^m,\nu P_{N,t}\pf\,.
\end{multline*}
\end{proof}

}

We now turn to the continuous-time limit of the non-linear chain $(Y_k)_{k\in\N}$.

\nv{
\begin{cor}\label{Cor-concl3}
There exists $C_6>0$ such that for all $\eta_0\in\mathcal P(\T^d)$ and all $\gamma \in(0,\gamma_0]$, if $(\eta_n)_{n\in\N}$ is such as defined in Section \ref{SubSec-DefProcess}, and $(\overline \eta_t)_{t\geqslant 0}$ is such as defined in Section \ref{Sec-continu} (with $\overline \eta_0 = \eta_0$), then
\[\mathcal W_{1}\po \eta_1, \overline \eta_{\gamma}\pf \ \leqslant \ C_6 \gamma^{3/2}\nvd{\,,}\]
\nvd{and for all $m\geqslant 1$,}
\[\mathcal W_{\rho} \po \eta_{m},\overline \eta_{m\gamma}\pf \ \leqslant \   \sqrt\gamma  C_6  \gamma \sum_{s=1}^m(1-\gamma\kappa)^{s-1}\,.\]
\end{cor}
\begin{proof}
For the first inequality, we could follow the proof of Lemma \ref{Lem-continu}, but, using the notations of the introduction, we will rather use the fact that
\[\eta_1 = \mathcal Law\po \tilde Z_1\ |\ \tilde T> \gamma\pf\,,\qquad \overline \eta_\gamma = \mathcal Law\po Z_\gamma\ |\ T>\gamma\pf\,,\]
where the gaussian variable $G_0$ in \eqref{Eq-EulerSchemeEDS} is equal to $B_{\gamma}/\sqrt\gamma$ where $(B_t)_{t\geqslant 0}$ is the Brownian motion involved in \eqref{Eq-EDS}, and $T$ and $\tilde T$ are defined with the same $E\sim \mathcal E(1)$. Recall the estimate \eqref{Eq-EulerEstimate} for the error from an Euler scheme to its initial diffusion. Then we bound
\begin{eqnarray*}
\mathbb E\po |\tilde Z_1 - Z_\gamma|\ |\ T>\gamma,\ \tilde T>\gamma\pf & \leqslant  & \po \mathbb P \po T>\gamma,\ \tilde T>\gamma\pf\pf^{-1}  \mathbb E\po |\tilde Z_1 - Z_\gamma|\pf \\
& \leqslant & \po 1 - e^{-\gamma_0 \|\lambda\|_\infty}\pf^{-1} c\gamma^{3/2}\,,
\end{eqnarray*}
which concludes the first part of the corollary.

For the second part, denoting $r_m = \mathcal W_{\rho} \po \eta_{m},\overline \eta_{m\gamma}\pf $, we bound
\begin{eqnarray*}
r_m & \leqslant & \mathcal W_{\rho} \po \eta_{m},\overline \eta_{m-1} Q_{\overline \eta_{m-1}}\pf + \mathcal W_{\rho} \po \overline \eta_{m-1} Q_{\overline \eta_{m-1}},\overline \eta_{m\gamma}\pf\\
& \leqslant & \po 1 - \gamma\kappa\pf r_{m-1} + C_6 \gamma^{3/2}\,,
\end{eqnarray*}
where we used the first part of the corollary and Corollary \ref{Cor-nonlineairetempslong}. An induction concludes.
\end{proof}
}

We can now prove propagation of chaos results for the continuous-time process:

\begin{cor}\label{Cor-chaos-continu}
For all $N\in\N$, $k\in\cco 1,N\ccf$ and all $t\geqslant 0$, if $(\bX_t)_{t\geqslant 0}$ is a Markov process with initial distribution $\eta_0^{\otimes N}$ associated to the semigroup $(P_{N,t})_{t\geqslant0}$ then, first,
\begin{eqnarray*}
\mathcal W_{\rho_k} \po \mathcal Law(\overline X_{1,t},\dots,\overline X_{k,t}),\overline{\eta}_t^{\otimes k}\pf & \leqslant & C_2 k \alpha(N) \int_0^t e^{-\kappa s}\dd s\,,
\end{eqnarray*}
and second, 
\begin{eqnarray*} 
\mathbb E\po \mathcal W_\rho \po \pi(\bX_t),\overline \eta_t\pf \pf & \leqslant & C_3 \alpha(N) \po 1 +  \int_0^t e^{-\kappa s}\dd s\pf\,.
\end{eqnarray*}
\end{cor}
\begin{proof}
As shown in the proof of Proposition \ref{Prop-chaos}, if $(\X,\Y)$ is an optimal coupling of $\mu$ and $\nu$,
\[\mathbb E\po \mathcal W_{\rho}\po \pi(\X),\pi(\Y)\pf \pf \ \leqslant \ \frac1N \mathcal W_{\rho_N}(\mu,\nu)\,.\]
Thus, considering  a time step $\gamma = t/m$, $m\in\N$, we decompose
\[\mathcal W_\rho \po \pi(\bX_t),\overline \eta_t\pf  \ \leqslant\  \mathcal W_\rho \po \pi(\bX_t),\pi(\X_{m})\pf  + \mathcal W_\rho \po \pi(\X_m), \eta_m\pf  + \mathcal W_\rho \po \eta_m,\overline \eta_t\pf\,, \]
take the expectation, apply Propositions \ref{Prop-chaos} and \ref{Prop-continu} and Corollary \ref{Cor-concl3}  and let $m$ go to infinity. This proves the second point, and the proof of the first one is similar, with  Corollary \ref{Cor-Chaos}.
\end{proof}

Up to now, we have sent either $N$ or $\gamma$ to their limit. When $\kappa>0$, if we let $t=m\gamma$ go to infinity at fixed $N$ and $\gamma$, we recover results on the equilibria of the processes. Indeed, note that Corollary \eqref{Cor-nonlineairetempslong} together with the Banach fixed-point theorem imply that $n\mapsto \eta_n$ admits a limit which is independent from $\eta_0$. Together with Proposition \ref{Prop-EgaliteLoi}, this is the unique QSD of the Markov chain \eqref{Eq-EulerSchemeEDS}. Denote it $\nu_{\gamma}$. Similarly, Proposition \ref{Prop-TempsLong} implies that $R_{N,\gamma}$ admits a unique invariant measure. Denote it $\mu_{\infty,N,\gamma}$, and $\mu_{\infty,N,\gamma}^{(k)}$ its first $kd$-dimensional marginal for $k\in\cco 1,N\ccf$ (i.e. the law of $(X_1,\dots,X_k)$ if $\X\sim\mu_{\infty,N,\gamma}$). Third, Corollary~\ref{Cor-tempslong-continu} implies that $(P_{N,t})_{t\geqslant 0}$ admits a unique invariant measure $\overline \mu_{\infty,N}$.

\begin{cor}\label{Cor-concl2}
If $\kappa>0$, then for all $N\in\N$ and $\gamma\in(0,\gamma_0]$
\begin{eqnarray*}
\mathcal W_{\rho_N} \po \mu_{\infty,N,\gamma}, \overline \mu_{\infty,N}\pf & \leqslant & \sqrt\gamma N \kappa^{-1} C_5  \,,
\end{eqnarray*}
\end{cor}

\begin{cor}\label{Cor-concl1}
If $\kappa>0$, then for all $N\in\N$, $k\in\cco 1,N\ccf$ and $\gamma\in(0,\gamma_0]$, first,
\begin{eqnarray*}
\mathcal W_{\rho_k} \po \mu_{\infty,N,\gamma}^{(k)},\nu_{\gamma}^{\otimes k}\pf & \leqslant & \kappa^{-1} C_2 k \alpha(N) \,,
\end{eqnarray*}
and second, 
\begin{eqnarray*}
\mathbb E_{\mu_{\infty,N,\gamma}}\po \mathcal W_\rho \po \pi(\X),\nu_{\gamma}\pf \pf & \leqslant & \kappa^{-1} C_3 \alpha(N) \,.
\end{eqnarray*}
\end{cor}

\begin{proof}[Proofs of Corollaries \ref{Cor-concl2} and \ref{Cor-concl1}]
Considering any $\eta_0\in\mathcal P(\T^d)$ and $m\in\N$,
\begin{multline*}
\mathcal W_{\rho_N} \po \mu_{\infty,N,\gamma}, \overline \mu_{\infty,N}\pf \ \leqslant \  \mathcal W_{\rho_N} \po \mu_{\infty,N,\gamma}, \eta_0^{\otimes N} R^m\pf  + \mathcal W_{\rho_N} \po \eta_0^{\otimes N} R^m, \eta_0^{\otimes N} P_{\gamma m}\pf\\ + \mathcal W_{\rho_N} \po \eta_0^{\otimes N} P_{\gamma m}, \overline \mu_{\infty,N}\pf 
\end{multline*}
Apply Proposition \ref{Prop-TempsLong} with $\mu = \mu_{\infty,N,\gamma}$ and $\nu = \eta_0^{\otimes N}$, Corollary \ref{Cor-tempslong-continu} with the same $\nu$ and with $\mu = \overline \mu_{\infty,N}$, and Proposition \ref{Prop-continu}. Letting $m$ go to infinity concludes the proof of Corollary \ref{Cor-concl2}. The proof of Corollary \ref{Cor-concl1} is similar (based on Proposition \ref{Prop-chaos} and Corollary~\ref{Cor-Chaos}, like Corollary \ref{Cor-chaos-continu}).
\end{proof}

Next, we can send two parameters to their limit. Sending $N$ to infinity and $\gamma$ to zero, we get the long time convergence of the non-linear process $(\overline Y_t)_{t\geqslant 0}$ introduced in Section~\ref{Sec-continu} (or, equivalently, of the process $Z$ solving \eqref{Eq-EDS} conditionned not to be dead):

\begin{cor}\label{Cor-tempslongContinuNonLin}
Let $(\overline \eta_t)_{t\geqslant 0}$ be such as defined in Section \ref{SubSec-DefProcess}, and $(\hat \eta_t)_{t\geqslant 0}$ be similarly defined but with a different initial distribution $\hat\eta_0 \in \mathcal P(\T^d) $. For all $t\geqslant 0$,
\[ \mathcal W_{\rho}\po \overline \eta_t,\hat \eta_t\pf \ \leqslant \ e^{-\kappa t} \mathcal W_{\rho}(\overline \eta_0,\hat\eta_0)\,.\]
\end{cor}
\begin{proof}
Thanks to Corollary \ref{Cor-concl3}, let $\gamma=t/m$ vanish in Corollary \ref{Cor-nonlineairetempslong}.
\end{proof}

In particular, if $\hat \eta_0$ is the QSD $\nu_*$ , by definition, $\hat \eta_t = \nu_*$ for all $t\geqslant 0$, so that Corollary~\ref{Cor-tempslongContinuNonLin} yields the uniqueness of the QSD and the exponential convergence of $\mathcal Law(Z_t\ |\ T>t)$ toward $\nu_*$ (which is a result in the spirit of \cite{ChampVilK2018,CV2020,DelMoralVillemonais2018,Bansaye}).

Now, at a fixed $\gamma> 0$, letting $t$ and $N$ go to infinity, we obtain an error bound between the QSD $\nu_*$ of the continus process \eqref{Eq-EDS} and the QSD $\nu_\gamma$ of the discrete scheme.
\begin{cor}\label{Cor-QSDgamma}
If $\kappa>0$, then for all  $\gamma\in(0,\gamma_0]$
\begin{eqnarray*}
\mathcal W_{\rho} \po \nu_{\gamma}, \nu_*\pf & \leqslant & \sqrt\gamma \kappa^{-1} C_6  \,,
\end{eqnarray*}
\end{cor}
\begin{proof}
Thanks to Corollaries \ref{Cor-nonlineairetempslong} and \ref{Cor-tempslongContinuNonLin} (applied with one of the initial condition being the equilibrium), let $m$ go to infinity in Corollary \ref{Cor-concl3}.
\end{proof}

Finally, letting $\gamma$ vanish and $t$ go to infinity at a fixed $N\in\N$, we obtain a propagation of chaos result at stationarity (as established \nv{first in \cite{Asselah}, and more recently with a CLT} in  \cite{LelievreReygner} in the case of a finite state space) for the continuous time system of interacting particle $(\bX_t)_{t\geqslant 0}$ introduced in Section \ref{Sec-continu}.

\begin{cor}\label{Cor-PropChaosEq}
If $\kappa>0$ and if $\bX$ is a random variable with law $\overline \mu_{\infty,N}$, then for all $N\in\N$ and $k\in\cco 1,N\ccf$, 
\begin{eqnarray*}
\mathcal W_{\rho_N} \po \mathcal Law(X_{1},\dots,X_{k}),\nu_*^{\otimes k}\pf & \leqslant & \kappa^{-1} C_2 k \alpha(N) \,,
\end{eqnarray*}
and second, 
\begin{eqnarray*} 
\mathbb E\po \mathcal W_\rho \po \pi(\bX),\nu_*\pf \pf & \leqslant & C_3 \alpha(N) \po 1 +  \kappa^{-1}\pf\,.
\end{eqnarray*}
\end{cor}
\begin{proof}
The proof is similar to Corollary \ref{Cor-chaos-continu}, letting $t$ go to infinity in  Corollary \ref{Cor-chaos-continu} thanks to Corollaries  \ref{Cor-tempslong-continu} and \ref{Cor-tempslongContinuNonLin}.
\end{proof}

\nvdd{Remark that our results at stationarity (Corollaries~\ref{Cor-concl2}, \ref{Cor-concl1}, \ref{Cor-QSDgamma} and \ref{Cor-PropChaosEq}) all require the perturbative condition $\kappa>0$. Yet, propagation of chaos at stationarity  for the continuous time process follows from the works \cite{Asselah,LelievreReygner,MicloDelMoral,DelMoralGuionnet,Rousset} in a much broader (non-perturbative) framework (and, although it doesn't seem to have been studied yet,  the situation should be similar for error bounds in $\gamma$ rather than $N$). As discussed in Section~\ref{Sec:previous_works}, error bounds on $N$ (and  possibly $\gamma$) that are uniform in time can be obtained thanks to the long-time convergence of the limit ($N=+\infty$) non-linear process. In our case, when $\kappa>0$, this long-time convergence  follows from the (uniform in $N$) long-time convergence of the particle system (whether is is possible to obtain the latter from the former is unclear), but it holds in more general cases (see \cite{ChampVilK2018,CV2020,DelMoralVillemonais2018,Bansaye} and references within) and, in those cases, results similar  to Corollaries~\ref{Cor-concl2}, \ref{Cor-concl1}, \ref{Cor-QSDgamma} and \ref{Cor-PropChaosEq} should hold. This question is out of the scope of the present work.}

\bigskip

\nv{
All our results can  be summarised in the following diagram :
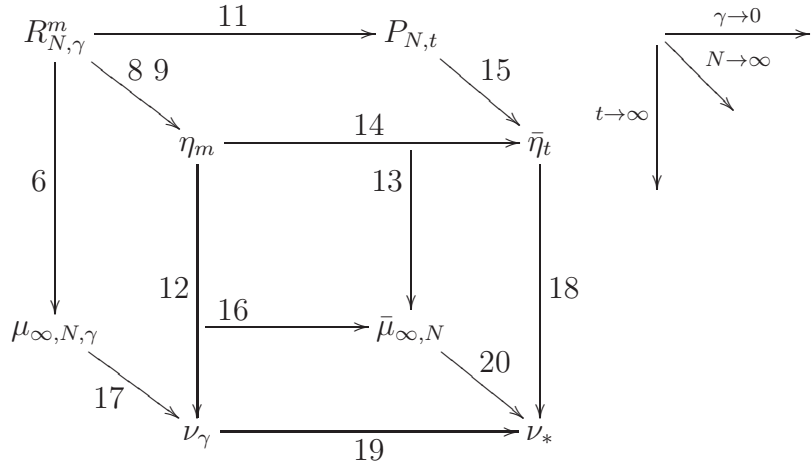
\begin{figure}[h]
\begin{center}
\[
\xymatrix{
    R_{N,\gamma}^m \ar[rrr]^{\ref{Prop-continu}} \ar[ddd]_{\ref{Prop-TempsLong}} \ar[dr]^{\ref{Prop-chaos}\ \ref{Cor-Chaos}} &&& P_{N,t} \ar[dr]^{\ref{Cor-chaos-continu}} \ar[ddd]_{\ref{Cor-tempslong-continu}} |!{[dl];[dr]}\hole \\
    & \eta_m \ar[rrr]^{\ref{Cor-concl3}} \ar[ddd]_{\ref{Cor-nonlineairetempslong}} &&& \bar{\eta}_t \ar[ddd]^{\ref{Cor-tempslongContinuNonLin}} \\
		& & & & \\
    \mu_{\infty,N,\gamma} \ar[rrr]^{\ref{Cor-concl2}} |!{[ur];[dr]}\hole \ar[dr]_{\ref{Cor-concl1}} &&& \bar{\mu}_{\infty,N} \ar[rd]^{\ref{Cor-PropChaosEq}} \\
    & \nu_{\gamma} \ar[rrr]_{\ref{Cor-QSDgamma}} &&& \nu_{\ast} \\
  }
\xymatrix{
 \ar[dd]_{t\rightarrow\infty} \ar[dr]^{N\rightarrow\infty} \ar[rr]^{\gamma \rightarrow 0} && \\ && \\ &&
 }
\]
\end{center}
\caption{Summary of the different results. The number on an arrow indicates the number of the Corollary or Proposition where the corresponding quantitative convergence is stated. Vertical, horizontal and diagonal arrows correspond respectively to $t$, $\gamma$ and $N$ going to their limit.}\label{Figure_Corollaires}
\end{figure}
}

Finally, we detail the proof of our main result.

\begin{proof}[Proof of Theorem \ref{Theorem-Main}]
For $\eta_0\in \mathcal P(\T^d)$, let $(\X,\Y)$ be an optimal coupling of $\mu_0 R^{\lfloor t/\gamma\rfloor}$ and $\eta_0^{\otimes N} R^{\lfloor t/\gamma\rfloor}$. As in the proof of Proposition \ref{Prop-chaos},
\[\mathbb E\po \mathcal W_{\rho}\po \pi(\X),\pi(\Y)\pf\pf \ \leqslant \ \frac1N \mathcal W_{\rho_N}\po \mu R^{\lfloor t/\gamma\rfloor},\eta_0^{\otimes N} R^{\lfloor t/\gamma\rfloor}\pf \ \leqslant \ a e^{-\kappa (t-\gamma_0)}\,,\]
where we used Proposition \ref{Prop-TempsLong} and the fact that $\rho_N(x,y)\leqslant Na$ for all $x,y\in\T^{dN}$. Then, by the triangular inequality,
\[\mathcal W_\rho\po \pi( \Y), \nu_*\pf \ \leqslant \ \mathcal W_\rho\po \pi( \Y), \eta_{\lfloor t/\gamma\rfloor}\pf + \mathcal W_\rho\po   \eta_{\lfloor t/\gamma\rfloor} , \nu_\gamma\pf + \mathcal W_\rho\po \nu_\gamma,\nu_*\pf\,.\]
Taking the expectation, applying Proposition \ref{Prop-chaos} and Corollaries \ref{Cor-nonlineairetempslong} (applied with $\tilde \eta_0 = \nu_\gamma$) and \ref{Cor-QSDgamma}, the boundedness of $\rho$ and the equivalence of $\mathcal W_\rho$ and $\mathcal W_1$ concludes.
\end{proof}

\subsection*{Acknowledgements}
Pierre Monmarch\'e thanks Bertrand Cloez for indicating the result stated in Proposition~\ref{Prop-EgaliteLoi}, and more generally for fruitful discussions. He acknowledges partial support by the projects EFI ANR-17-CE40-0030 and METANOLIN of the French National Research Agency.

\bibliographystyle{plain}
\bibliography{../biblio}

\end{document}